\documentclass{article}
\usepackage[OT2,OT1]{fontenc}
\usepackage{amsfonts,amsmath, amssymb,latexsym,mathrsfs}
\usepackage[all,cmtip]{xy}
\usepackage{float}
\def\SH{\mbox{\fontencoding{OT2}\selectfont\char88}}

\def\Z{{\mathbb Z}}
\def\A{{\mathbb A}}

\def\SL{{\rm SL}}
\def\GL{{\rm GL}}
\def\SO{{\rm SO}}
\def\PGL{{\rm PGL}}

\def\inv{{\rm inv}}
\def\Stab{{\rm Stab}}
\def\Inv{{\rm Inv}}

\def\P{{\mathbb P}}

\def\Det{{\rm Det}}
\def\Aut{{\rm Aut}}
\def\irr{{\rm irr}}

\def\r{{\rm r}}
\def\Vol{{\rm Vol}}
\def\R{{\mathbb R}}
\def\F{{\mathbb F}}
\def\FF{{\mathcal F}}
\def\RR{{\mathcal R}}
\def\Q{{\mathbb Q}}
\def\H{{\mathcal H}}
\def\J{{\mathcal J}}

\def\Z{{\mathbb Z}}
\def\P{{\mathbb P}}
\def\F{{\mathbb F}}
\def\Q{{\mathbb Q}}
\def\C{{\mathbb C}}
\def\H{{\mathcal H}}
\def\W{{\mathcal W}}

\def\wzn2{{W_{\Z,+}^{(2-)}}}

\def\fz1{{F_{\Z,1}}}

\newtheorem{theorem}{Theorem}[section]
\newtheorem{corollary}[theorem]{Corollary}
\newtheorem{lemma}[theorem]{Lemma}
\newtheorem{remark}[theorem]{Remark}
\newtheorem{proposition}[theorem]{Proposition}
\newenvironment{proof}{\noindent {\bf Proof:}}{$\Box$ \vspace{2 ex}}

\begin{document}

\title{Binary quartic forms having bounded invariants, and \\the
  average rank of elliptic curves} 
\title{On the boundedness of the 
  average rank of elliptic curves} 
\title{Binary quartic forms having bounded invariants, and \\the boundedness of
 the average rank of elliptic curves} 

\author{Manjul Bhargava and Arul  Shankar}

\maketitle

\begin{abstract}
  We prove a theorem giving the asymptotic number of binary quartic
  forms having bounded invariants; this extends, to the quartic case,
  the classical results of Gauss and Davenport in the quadratic and
  cubic cases, respectively. Our techniques are quite general, and may
  be applied to counting integral orbits in other representations of
  algebraic groups. 

  We use these counting results to prove that the average rank of
  elliptic curves over $\Q$, when ordered by their heights, is
  bounded.  In particular, we show that when elliptic curves are
  ordered by height, the mean size of the $2$-Selmer group is
  $3$. This implies that the limsup of the average rank of elliptic
  curves is at most~$1.5$.
\end{abstract}

\setcounter{tocdepth}{4}


\section{Introduction}
\subsection{Average ranks of elliptic curves}
Any elliptic curve $E$ over $\Q$ is isomorphic to a unique curve of
the form $E_{A,B}:y^2=x^3+Ax+B$, where $A,B \in \Z$ and for all primes
$p$:\, $p^6 \nmid B$ whenever $p^4 \mid A$. Let $H(E_{A,B})$ denote
the (naive) $height$ of $E_{A,B}$, defined by $H(E_{A,B}):= \max
\{4|A^3|,27B^2\}$. Let $\Delta (E_{A,B})$ and $C(E_{A,B})$ denote the
discriminant and conductor of $E_{A,B}$, respectively.

It is an old conjecture, originating in works of Goldfeld~\cite{G1}
and Katz-Sarnak~\cite{KS}, that a density of $50\%$ of all elliptic
curves over $\Q$ have rank $0$ and $50\%$ have rank $1$. These
densities are expected to hold true regardless of whether one orders
curves by height, discriminant, or conductor. In particular, one
expects the average rank of all elliptic curves to be $1/2$. However,
it has not previously been known that the average rank of all elliptic
curves is even {\it finite} (i.e., bounded).  Computations have also
not been very helpful in this regard; see \cite{BMSW} for a nice
survey.

In \cite{AB}, Brumer showed that the generalized Riemann hypothesis
and the Birch--Swinnerton-Dyer conjectures together imply that the
average rank of all elliptic curves, when ordered by their heights, is
finite and is in fact bounded above by $2.3$. Still assuming the
generalized Riemann hypothesis and the Birch--Swinnerton-Dyer
conjectures, this constant was subsequently improved to $2$ by
Heath-Brown \cite{HB} and to $25/14 \sim 1.79$ by Young \cite{MY}.

The purpose of this article is to prove unconditionally that the
average rank of all elliptic curves, when ordered by their heights, is
finite. In fact, we prove the same for the $2$-{\it Selmer rank}.
Recall that the $2$-Selmer group $S_2(E)$ of an elliptic curve $E$
over $\Q$ fits into an exact sequence
\begin{equation}\label{exact}
0 \to E(\Q)/2E(\Q) \to S_2(E) \to \SH_E[2] \to 0, 
\end{equation}
where $\SH_E[2]$ denotes the $2$-torsion subgroup of the
Tate-Shafarevich group $\SH_E$ of $E$. The 2-Selmer group is an
elementary abelian 2-group of order $2^s$ for some integer
$s\geq 0$, and the quantity $s$ is called the {\it $2$-Selmer rank of $E$}.
Thus the $2$-Selmer rank of $E$ gives an upper bound for the rank of $E$.

Our main theorem on the 2-Selmer group is as follows:
\begin{theorem}\label{mainellip}
When all elliptic curves $E/\Q$ are ordered by height,
the average size of the $2$-Selmer group $S_{2}(E)$ is $3$.
\end{theorem}

We immediately conclude that:

\begin{corollary}\label{corellip}
  When all elliptic curves over $\Q$ are ordered by height, their average
  $2$-Selmer rank is at most $1.5$; thus their average rank is also
  at most $1.5$.
\end{corollary}
Indeed, note that Equation (\ref{exact}) implies that
\begin{equation}\label{sum}
r_2(S_{2}(E)) = r(E) + r_2(E(\Q)[2]) + r_2(\SH_E[2]) \,,
\end{equation}
where we have used $r(E)$ to denote the rank of $E$ and $r_2(G)$
(for an elementary abelian 2-group $G$) to denote
$\dim_{\F_2}(G)$.  Due to the inequality $2r_2(S_2(E))\leq 2^{r_2(S_2(E))}=
|S_2(E)|$, Theorem~\ref{mainellip} bounds the mean of the left hand side
of (\ref{sum}) by 1.5, and thus the same bound holds also for the average
size of each of the terms on the right hand side of (\ref{sum}).  In
particular, the average size of $r_2(\SH_E[2])$ is also at most
1.5. Meanwhile, it is elementary 
that the mean size of $r_2(E(\Q)[2])$ is 0, i.e., 
$0\%$ of elliptic curves possess rational 2-torsion.

We will in fact prove a stronger version of Theorem \ref{mainellip},
namely:
\begin{theorem}\label{ellipcong}
  When elliptic curves $E:y^2=x^3+Ax+B$, in any family defined by finitely
  many congruence conditions on the coefficients $A$ and $B$, are ordered by
  height, the average size of the $2$-Selmer group $S_2(E)$ is~$3$.
\end{theorem}
Thus the average size of the 2-Selmer group remains 3 even when one
averages over any subset of elliptic curves defined by finitely many 
congruence conditions.  We will actually prove Theorem~\ref{ellipcong} 
for an even larger class of families, including some that are defined by
certain natural {\it infinite} sets of congruence conditions.

We note that the boundedness of the average rank of elliptic curves
has been known previously in certain special one-parameter families of
elliptic curves.  For example, in \cite{Fouvry}, Fouvry shows that the
average rank is bounded in the family of cubic twists $y^2 = x^3+k$ as
$k$ varies.  In \cite{Heathbrown}, Heath-Brown shows that the average
rank is bounded for the family of ``congruent number curves''
$y^2=x^3-d^2x$ as $d$ varies, and in fact he determines the exact
distribution of 2-Selmer ranks, which implies that the average size of
the 2-Selmer group in this family is 3.  In more recent work,
Swinnerton-Dyer \cite{SD} and Kane \cite{Kane} have proven that the
same distributions hold for any family of quadratic twists of a single
curve with full rational $2$-torsion.  Our Theorem 1.1 shows that, as
far as 2-Selmer ranks are concerned, general elliptic curves seem to
behave, on average, in a manner similar to curves in a family of
twists.  

In the function field case, the boundedness of the average rank of all
elliptic curves was proven by de Jong~\cite{deJong}, who showed that
for a finite field of characteristic not equal to 3, the average size
of the 3-Selmer group of all elliptic curves over $\F_q(t)$ is bounded
(and is in fact at most $4+\varepsilon(q)$ for an explicit function
$\varepsilon(q)$ that tends to 0 as $q\to\infty$).  Our main result,
Theorem~1.1, may be viewed as a precise version of de Jong's theorem
over the number field $\Q$, with the 3-Selmer group replaced by the
2-Selmer group.  We will treat the case of the 3-Selmer group over
$\Q$ in a forthcoming article.

Theorems 1.1 and 1.3 also confirm two remarkable sets of heuristics in
the literature.  In \cite{Delaunay}, Delaunay used a
Cohen--Lenstra-style model to conjecture the distribution of the
Tate-Shafarevich group of elliptic curves.  Delaunay's heuristics,
coupled with the rank distribution conjecture of Goldfeld and
Katz--Sarnak, imply that the average size of the 2-Selmer group is 3.
More recently, by a completely different approach, Poonen and Rains
\cite{PR} model the Selmer group as a random intersection of
isotropic subspaces of a quadratic space, and again, they predict that
the average size of the 2-Selmer group should be 3. These heuristics
thus give an interpretation for the number 3 that appears in Theorems
1.1 and 1.3.  For a further interpretation of the number 3 in terms
of local masses of 2-coverings of elliptic curves and the 
Tamagawa number of~$\PGL_2$, see Sections 3.3 and 3.6.

\vspace{.1in}
\subsection{Counting binary forms having bounded invariants
  (particularly quartic forms)}

We prove the above theorems by developing techniques to count integral
orbits, having bounded invariants, in certain {coregular
  representations} over $\Z$.  We define a {\it coregular
  representation} as a pair $(G,V)$, where $G$ is an algebraic group
and $V$ is a representation of $G$ (for our purposes, both defined
over~$\Z$) such that the ring of relative polynomial invariants of
$G(\C)$ on $V(\C)$ is a polynomial ring.  \pagebreak
Although our techniques are quite
general, in this article we concentrate primarily on the case where
$G=\GL_2$ and $V$ is the space of {\bf binary quartic forms}
$ax^4+bx^3y+cx^2y^2+dx^3y+ey^4$.

\vspace{.05in} The problem of counting integral binary forms having
bounded invariants is a classical one.  The case of binary quadratic
forms was first treated in the influential work {\it Disquisitiones
  Arithmeticae} of Gauss in 1801.  Gauss studied the action of
$\SL_2(\Z)$ on the space of integral binary quadratic forms
$f(x,y)=ax^2+bxy+cy^2$ ($a,b,c\in\Z$)\footnote{Gauss actually
  considered only forms where $b$ is even; however, from the modern
  point of view, it is natural to allow all three coefficients $a,b,c$
  to be arbitrary integers.} via linear substitution of variable, in
terms of the unique polynomial invariant for this action, namely the
discriminant $\Delta(f)=b^2-4ac$.  (The polynomial invariant
$\Delta(f)$ is ``unique'' in the sense that the ring of polynomial
invariants is generated by one element, namely $\Delta(f)$.)

Gauss conjectured, and Mertens~\cite{Mertens} and Siegel~\cite{Siegel}
proved, respectively, that:

\begin{theorem}[Mertens 1874/Siegel 1944]\label{GMS}
Let $h_D$ denote the number of $\SL_2(\Z)$-equivalence classes of
irreducible integral binary quadratic forms having discriminant $D$.
Then:
\begin{itemize}
\item[{\rm (a)}]
$\displaystyle{\quad\,\,\sum_{-X<D<0}h_D\,\sim\,\, \frac\pi{18}\cdot X^{3/2}}$;
\item[{\rm (b)}]
$\,\displaystyle{\sum_{0<D<X}h_D\log\;\!\varepsilon_D \,\sim\,\, \frac{\pi^2}{18}
\cdot X^{3/2}};$
\end{itemize}
here $\varepsilon_D=(t+u\sqrt{D})/2$, where $t,u$ are the smallest 
positive integral solutions of $t^2-Du^2=4$.
\end{theorem}
Note that $h_D$ and $\log\,\epsilon_D$ have important algebraic
number theoretic interpretations, namely, $h(D)$ is the (narrow) class number
and $\log\,\epsilon_D$ is the regulator of 
the unique quadratic order of discriminant $D$.  
Thus Theorem~\ref{GMS}(a) gives the average size of the class number
of imaginary quadratic orders up to a given absolute discriminant,
while (b) gives the average size of the class number times the
regulator of real quadratic orders up to a given discriminant.

The next natural case to consider is that of integral binary cubic
forms $f(x,y)=ax^3+bx^2y+cxy^2+dy^3$ ($a,b,c,d\in\Z$).  The group
$\GL_2(\Z)$ (or $\SL_2(\Z)$) again naturally acts on such forms, and
there is again a unique polynomial invariant for this action, namely,
the discriminant
$$\Delta(f)=b^2c^2+18abcd-4ac^3-4b^3d-27a^2d^2.$$
The question, as in the case of binary quadratic forms, is: how many
classes $h(D)$ of irreducible binary cubic forms are there with
discriminant $D$, on average, as $D$ varies?

This question was first answered by Davenport~\cite{Davenport2}:
\begin{theorem}[Davenport 1951]\label{dav}
Let $h(D)$ denote the number of $\GL_2(\Z)$-equivalence classes of
irreducible integral binary cubic forms having discriminant $D$.
Then:
\begin{itemize}
\item[{\rm (a)}]
$\,\displaystyle{\sum_{-X<D<0}h(D) \,\sim\,\, \frac{\pi^2}{24}\cdot X}$; 
\item[{\rm (b)}]
$\,\displaystyle{\,\,\sum_{0<D<X}\,\;\!h(D) \,\sim\,\, \frac{\pi^2}{72}\cdot X}.$
\end{itemize}
\end{theorem}
Davenport's theorem thus states that the number of equivalence classes
of irreducible binary cubic forms per discriminant is a constant on
average.  This too has an important algebraic number theoretic
interpretation.  Since equivalence classes of irreducible integral 
binary cubic forms are in bijection
with orders in cubic fields (see Delone--Faddeev's work~\cite{DF}),
Theorem~\ref{dav} states that there are a constant number of
(isomorphism classes) of cubic orders per discriminant, on average.
Davenport's theorem was an essential ingredient in the classical work
of Davenport and Heilbronn on the density of discriminants
of cubic fields (see~\cite{DH}).

The next natural case to consider is that of binary quartic forms.
The group $\GL_2(\Z)$ again acts on the space of binary quartic forms
$f(x,y)=ax^4+bx^3y+cx^2y^2+dxy^3+ey^4$ ($a,b,c,d,e\in\Z$) by linear
substitution of variable.  Note that in each of the cases of binary
quadratic and binary cubic forms, the ring of invariants was generated
by one element. Binary quartic forms historically have been more
difficult to treat because the ring of invariants is now generated by
two independent invariants, traditionally denoted $I$ and $J$.  For
$f(x,y)$ as above, we have the following explicit formulae for these
invariants:
$$
\begin{array}{c}
I(f)=12ae-3bd+c^2,\\[.05in]
J(f)=72ace+9bcd-27ad^2-27eb^2-2c^3.
\end{array}
$$
Any other polynomial invariant for the action of $\GL_2(\Z)$ on binary
quartic forms can be expressed as a polynomial in these invariants;
for example, the discriminant $\Delta(f)$ of a binary quartic form can
be expressed in terms of $I(f)$ and $J(f)$ as follows:
$$\Delta(f):=\Delta(I(f),J(f)):=(4I(f)^3-J(f)^2)/27.$$

It follows from work of Borel and Harish-Chandra~\cite[Theorem
6.9]{BH} that the number of equivalence classes of integral binary
quartic forms, having any given fixed values of $I$ and $J$ (so long
as $I$ and $J$ are not both equal to zero), is
finite.\footnote{It is also true that the number of equivalence
  classes of binary quartic forms having a fixed nonzero value of the
  single invariant $\Delta(f)=\frac{1}{27}(4I(f)^3-J(f)^2)$ is finite,
  since the set of integral points on the elliptic curve
  $4x^3-y^2=27d$ is finite for each $d\neq 0$. However, the latter
  fact will not be used here.}  This raises the question as to how
many classes $h(I,J)$ of irreducible binary quartic forms with
invariants $I,J$ are there, on average, as the pair $(I,J)$ varies?

To answer this question, we require just a bit of notation.  
Let us define the (naive) {\it height} of $f(x,y)$ by
$H(f):=H(I,J):=\max\{|I^3|,J^2/4\}$ (the constant $1/4$ on $J^2$ is
present for convenience, and is not of any real importance).  Thus
$H(f)$ is a ``degree 6'' function on the coefficients of $f$, in the
sense that $H(rf)=r^6H(f)$ for any constant $r$.  We prove:

\begin{theorem}\label{bqcount}
  Let $h^{(i)}(I,J)$ denote the number of $\GL_2(\Z)$-equivalence classes
  of irreducible binary quartic forms having $4-2i$ real roots in $\P^1$
  and invariants equal to $I$ and $J$. Then:
\begin{itemize}
\item[{\rm (a)}] $\displaystyle{\sum_{\substack{H(I,J)<X}}h^{(0)}(I,J)
\,=\,\frac{4}{135}\zeta(2)X^{5/6}+O(X^{3/4+\epsilon})\,;}$
\item[{\rm (b)}] $\displaystyle{\sum_{\substack{H(I,J)<X}}h^{(1)}(I,J)
\,=\,\frac{32}{135}\zeta(2)X^{5/6}+O(X^{3/4+\epsilon})\,;}$
\item[{\rm (c)}] $\displaystyle{\sum_{\substack{H(I,J)<X}}h^{(2)}(I,J)
\,=\,\frac{8}{135}\zeta(2)X^{5/6}+O(X^{3/4+\epsilon})\,.}$
\end{itemize}
\end{theorem}

In order to obtain the average size of $h^{(i)}(I,J)$, as $(I,J)$
varies, we first wish to know which pairs $(I,J)$ can actually occur
as the invariants of an integral binary quartic form.  In the
quadratic and cubic cases, this is easy and well-known: a number
occurs as the discriminant of a binary quadratic (resp.\ cubic) form
if and only if it is congruent to $0$ or $1$~(mod 4).  

In the binary quartic case, we prove that a similar scenario occurs,
namely, an $(I,J)$ is {\it eligible}---i.e., it occurs as the
invariants of some integral binary quartic form---if and only if it
satisfies any one of a certain specified finite set of
congruence conditions modulo 27.  More precisely, we prove:

\begin{theorem}\label{eligible}
  A pair $(I,J)\in\Z\times\Z$ occurs as the invariants of an integral
  binary quartic form if and only if it satisfies
  one of the following congruence conditions:
\begin{itemize}
\item[{\rm (a)}] $I \equiv 0 \pmod3$ and $J \equiv 0 \pmod{27},$

\item[{\rm (b)}] $I\equiv 1 \pmod9$ and $J \equiv \pm 2 \pmod{27},$

\item[{\rm (c)}] $I\equiv 4 \pmod9$ and $J \equiv \pm 16 \pmod{27},$

\item[{\rm (d)}] $I\equiv 7 \pmod9$ and $J \equiv \pm 7 \pmod{27}.$
\end{itemize}

\end{theorem}

It follows that the number of eligible $(I,J)$, with $H(I,J)<X$, is a
constant times $X^{5/6}$; thus, by Theorem~\ref{bqcount}, the number
of classes of binary quartic forms per eligible $(I,J)$ is a finite
constant on average.  We have the following theorem:

\begin{theorem}\label{bqaverage}
  Let $h^{(i)}(I,J)$ denote the number of $\GL_2(\Z)$-equivalence
  classes of irreducible binary quartic forms having $4-2i$ real roots
  and invariants equal to $I$ and $J$. Let $n_0=4$, $n_1=2$, and
  $n_2=2$. Then, for $i=0,1,2$, we have:
$$\displaystyle\lim_{X\rightarrow\infty}
\displaystyle\frac{\displaystyle\sum_{H(I,J)<X}h^{(i)}(I,J)}
{\displaystyle\sum_{\substack{(I,J) \mbox{ {\rm \scriptsize{eligible}}
     }\\[.025in](-1)^i\Delta(I,J)>0\\[.025in]H(I,J)<X}}1}=\frac{2\zeta(2)}{n_i}.$$
\end{theorem}
Thus, Theorem~\ref{bqaverage} says that 
the number of equivalence
classes of binary quartic forms per eligible $(I,J)$, having a given
number of real roots, is a constant on average.
This constant is either $\zeta(2)/2$ or $\zeta(2)$, depending on
whether the given number of real roots is 4 or less than 4, respectively.

We in fact prove a strengthening of Theorem~\ref{bqcount}; namely, we
obtain the asymptotic count of binary quartic forms, having bounded
invariants, satisfying any specified finite set of congruence
conditions.  Such a modification will be crucial for the
applications to elliptic curves, which we discuss next. 

\subsection{Binary quartic forms and $2$-Selmer groups of elliptic curves}

To use the latter counting results involving binary
quartic forms to understand the average size of 2-Selmer groups of
elliptic curves (as in
Theorem~\ref{mainellip}), we recall that an element of the 2-Selmer
group of an elliptic curve $E/\Q$ may be thought of as a {``locally
  soluble 2-covering''}.  A {\it $2$-covering} of $E/\Q$ is a genus one
curve $C/\Q$ together with maps $\phi:C\to E$ and $\theta:C\to E$, where
$\phi$ is an isomorphism defined over $\C$, and $\theta$ is a degree~4
map defined over $\Q$, such that the following diagram commutes:
$$\xymatrix{E \ar[r]^{[2]} &E\\C\ar[u]^\phi\ar[ur]_\theta}$$
Thus a 2-covering $C=(C,\phi,\theta)$ may be viewed as a ``twist over $\Q$ of
the multiplication-by-2 map on $E$''.   
Two $2$-coverings $C$ and $C'$ are said to be {\it isomorphic} if
there exists an isomorphism $\Phi:C\to C'$ defined over~$\Q$, and a
2-torsion point $P\in E$, such that the following diagram
commutes:
$$\xymatrix{E \ar[r]^{+P} &E\\C\ar[u]^\phi\ar[r]_\Phi &C'\ar[u]_{\phi'}}$$
A {\it soluble $2$-covering} $C$ is one that possesses a rational
point, while a {\it locally soluble $2$-covering} $C$ is one that
possesses an $\R$-point and a $\Q_p$-point for all primes $p$. Then
we have natural bijections
\begin{eqnarray*} 
\{\mbox{\rm soluble 2-coverings}\}/\sim &\,\longleftrightarrow\,& E(\Q)/2E(\Q); \\
\{\mbox{\rm locally soluble 2-coverings}\}/\sim &\,\longleftrightarrow\,& S_{2}(E),
\end{eqnarray*}
giving each set on the left too the structure of a finite abelian $2$-group.

How does counting elements of $S_{2}(E)$ lead to counting binary
quartic forms?  There is a result of Birch and Swinnerton-Dyer (see
\cite[Lemma 2]{BSD}) that states that any locally soluble 2-covering
$C$ possesses a canonically associated degree 2 divisor defined over
$\Q$, thus yielding a double cover $C\to\P^1$ ramified at 4 points.
We thus obtain a binary quartic form over $\Q$, well-defined up to
$\GL_2(\Q)$-equivalence!  This connection between 2-Selmer group
elements and binary quartic forms was first introduced and used in the
original elliptic curve computations of Birch and Swinnerton-Dyer,
which led them to their celebrated conjecture.  Indeed, this
interpretation of binary quartic forms in terms of 2-Selmer groups is
still one of the fastest ways of computing and enumerating ranks of elliptic curves in
practice, as in, e.g., Cremona's influential {\tt mwrank} program.

We use this connection and the above counting results on binary
quartic forms to prove Theorems~\ref{mainellip} and
\ref{ellipcong}, as follows:

\begin{itemize}
\item Given $A,B\in\Z$, 
construct an {\it integral} binary quartic
  form $f$ for each element of $S_{2}(E_{A,B})$ such that
\begin{itemize}
\item $y^2=f(x)$ gives the desired 2-covering;
\item the invariants $(I(f),J(f))$ of $f$ agree with the invariants $(A,B)$
  of the elliptic curve (at least up to bounded powers of 2 and 3);
\end{itemize}
\item Count these integral binary quartic forms via congruence versions of Theorem~1.6.  The
  relevant binary quartic forms are actually defined by infinitely many
  congruence conditions, so a sieve has to be performed.

\item A uniformity estimate, which shows that the error term does not
  grow too large as more and more of the relevant congruence
  conditions are imposed, must be proven to perform this sieve.  This
  is perhaps the most technical ingredient in this work.
  It is accomplished by embedding the space of binary quartic forms
  into a certain larger space---namely, the space of pairs of ternary
  quadratic forms---where such uniformity estimates are more amenable
  and have been studied previously 
  in the context of counting quartic fields~\cite{dodqf}.
\end{itemize}

 This paper is organized as follows.  In Section 2, we
  study the distribution of $\GL_2(\Z)$-equivalence classes of binary
  quartic forms with respect to their fundamental invariants $I$ and
  $J$; in particular, we prove
  Theorems~\ref{bqcount}--\ref{bqaverage}.
We also prove the uniformity estimates that are necessary
  to count binary quartic forms satisfying our desired infinite sets
  of congruence conditions.  

   In Section 3, we describe the precise connection
  between binary quartic forms and elements in the 2-Selmer groups of
  elliptic curves.  This connection allows us, through the use of
  certain mass formulae for elliptic curves over $\Q_p$, to compute
  the average size of the 2-Selmer groups of elliptic curves (or of
  appropriate families of elliptic curves) via a count of binary
  quartic forms satisfying a certain weighted infinite
  set of congruence conditions.  We then apply the uniformity results
  of Section~2 to count these binary quartic forms, thus completing
  the proofs of Theorems~\ref{mainellip} and \ref{ellipcong}.



  \section{The number of classes of integral binary quartic forms
    having bounded invariants}

Let $V_\R$ denote the vector space of binary quartic
forms over the real numbers $\R$.  We express an element $f\in V_\R$
in the form $f(x,y)=ax^4+bx^3y+cx^2y^2+dxy^3+ey^4$, where $a,b,c,d,$ and $e$
are real numbers. Such an $f\in V_\R$ is said to be {\it integral} if
$a,b,c,d,e\in\Z$.

In this section, we derive asymptotics for the number of
$\GL_2(\Z)$-equivalence classes of irreducible integral binary quartic
forms having bounded invariants.  We also describe how these
asymptotics change when we restrict to counting those binary
quartic forms satisfying certain specified sets of congruence
conditions.  In particular, we prove
Theorems~\ref{bqcount}--\ref{bqaverage}.

The group $\GL_2(\R)$ naturally acts on $V_\R$; namely, an element $\gamma
\in \GL_2(\R)$ acts on $f(x,y)$ by linear substitution of variable:
\begin{equation}
  \label{equntwistedaction}
\gamma\cdot f(x,y)=f((x,y)\cdot \gamma).  
\end{equation}
This action of $\GL_2(\R)$ on $V_\R$ is a left action, i.e.,
$(\gamma_1\gamma_2) \cdot f=\gamma_1\cdot(\gamma_2\cdot f)$.

We also consider the action of $\SL_2^\pm(\R)$ on $V_\R$, where
$\SL_2^\pm(\R)\subset\GL_2(\R)$ is the subgroup of elements in
$\GL_2(\R)$ having determinant equal to $\pm 1$.  The ring of
invariants for this action is generated by two independent generators
of degrees 2 and 3 which are traditionally denoted by $I$ and $J$,
respectively. If $f(x,y)=ax^4+bx^3y+cx^2y^2+dxy^3+ey^4$, then
\begin{equation}
  \begin{array}{rcl}
    \displaystyle I(f)&=&12ae-3bd+c^2,\\[.02in]
    \displaystyle J(f)&=&72ace+9bcd-27ad^2-27eb^2-2c^3.
  \end{array}
\end{equation}
The quantities $I(f)$ and $J(f)$ are also {\it relative invariants}
for the action of $\GL_2(\R)$ on $V_\R$: we have 
\begin{equation}\label{eq412}
  \begin{array}{rcl}
I(\gamma \cdot f)&=&(\det\gamma)^4I(f),\\[.03in]
J(\gamma \cdot f)&=&(\det \gamma)^6J(f).
  \end{array}
\end{equation}
The discriminant
$\Delta(f)$ of a binary quartic form $f$, being a relative invariant
of degree $6$, can thus be expressed in terms of $I$ and $J$, namely,
$\Delta(f)=(4I(f)^3-J(f)^2)/27$.  We define the {\it height}
$H(f)$ of a binary quartic form $f$ by
\begin{equation}\label{hbqdef}
H(f):=H(I,J)=\max\{|I|^3,J^2/4\}.
\end{equation}

The action of $\GL_2(\Z)$ on $V_\R$ evidently preserves the lattice
$V_\Z$ consisting of the integral elements of $V_\R$, and so we may
ask: how many $\GL_2(\Z)$-classes of forms are there having height at
most $X$?  More precisely, we may ask: 
how many $\GL_2(\Z)$-classes of forms are there with height at
most $X$ and a given number of real roots? 

To this end, for $i=0$, $1$, and $2$, let $V_\Z^{(i)}$ denote the set
of elements in $V_\Z$ having nonzero discriminant and $i$ pairs of
complex conjugate roots and $4-2i$ real roots in $\P^1_\C$.  For any
$\GL_2(\Z)$-invariant set $S\subset V_\Z$, let $N(S;X)$ denote the
number of $\GL_2(\Z)$-equivalence classes of irreducible elements
$f\in S$ satisfying $H(f)<X$.  Then the main theorem of this section
is the following restatement of Theorem~\ref{bqcount}:

\begin{theorem}\label{refbq}
We have
\begin{itemize}
\item[{\rm (a)}] $N(V_\Z^{(0)};X)=\displaystyle\frac{4}{135}\zeta(2)X^{5/6}+O(X^{3/4+\epsilon});$
\item[{\rm (b)}] $N(V_\Z^{(1)};X)=\displaystyle\frac{32}{135}\zeta(2)X^{5/6}+O(X^{3/4+\epsilon});$
\item[{\rm (c)}] $N(V_\Z^{(2)};X)=\displaystyle\frac{8}{135}\zeta(2)X^{5/6}+O(X^{3/4+\epsilon}).$
\end{itemize}
\end{theorem}

Our strategy to prove Theorem~\ref{refbq} is as follows. In
\S\ref{redthbq}, we develop the necessary reduction theory
needed to establish convenient fundamental domains for the
action of $\GL_2(\Z)$ on $V_\R$. The primary difficulty in counting points in
these fundamental domains is that they are not bounded, but instead
have a rather complicated cuspidal region going off to infinity.  To
deal with and effectively handle this cusp, in
\S\ref{redsec} we investigate the
distribution of reducible and irreducible points inside these
fundamental domains. Specifically, we prove that the cusp contains
only reducible points, while the remainder of the domain outside the
cuspidal region contains primarily irreducible points.
In \S\ref{avgsec}, we develop a refinement of an averaging
method introduced in \cite{dodqf}, \cite{dodpf} to count
points in these fundamental regions in terms of the volumes of these
domains. The volumes of the fundamental regions are then computed in
\S\ref{secvol}, completing the proof of Theorem~\ref{refbq}.

In \S\ref{secbqcong}, we prove a stronger version of Theorem
\ref{refbq} where we restrict to counting those binary quartic forms
whose coefficients satisfy finitely many congruence conditions.  In
\S\ref{secunif}, we prove the necessary estimates that uniformly bound
the number of $\GL_2(\Z)$-orbits on binary quartic forms having
bounded height whose discriminants are divisible by the square of some
large prime.  In~\S\ref{squarefree}, we then describe how these
uniformity estimates allow one to count
the number of $\GL_2(\Z)$-orbits of binary quartic forms of bounded
height having {squarefree discriminant} (or satisfying other similar
sets of infinitely many congruence conditions).  We will require such
results when we prove Theorems~\ref{mainellip} and \ref{ellipcong} in
Section 3.

\subsection{Reduction theory}\label{redthbq}

For $i=0$, $1$, and $2$, let $V_\R^{(i)}$ denote the set of points in
$V_\R$ having nonzero discriminant and $i$ pairs of complex roots and $4-2i$ real roots in
$\P^1_\C$. 
Then $V_\R^{(2)}$
is the set of {\it definite} forms in $V_\R$, i.e., forms $f(x,y)$
that take only positive or only negative values when evaluated at
nonzero vectors $(x_0,y_0)\in\R^2$.  Let $V_\R^{(2+)}$ (resp.\
$V_\R^{(2-)}$) denote the subset of $V_\R^{(2)}$ consisting of the
{\it positive definite forms} (resp.\ {\it negative definite forms}).
Note that for $i=0$, $1,$ and $2$ we have $V_\Z^{(i)}=V_\R^{(i)}\cap
V_\Z$.  We analogously define $V_\Z^{(i)}=V_\R^{(i)}\cap V_\Z$ for
$i=2+$ and $2-$.

We then have the following facts (see \cite[Remark 2]{cremred}):
\begin{itemize}
\item[1.] The set of binary quartic forms in $V_\R$ having fixed invariants
  $I$ and $J$ consists of just one $\SL^\pm_2(\R)$-orbit if
  $4I^3-J^2<0$; this orbit lies in $V_\R^{(1)}$. 

\item[2.] The set of binary quartic forms in $V_\R$ having fixed invariants
  $I$ and $J$ consists of three $\SL^\pm_2(\R)$-orbits if
  $4I^3-J^2>0$; in that case, there is one such orbit from each of $V_\R^{(0)}$,
  $V_\R^{(2+)}$, and $V_\R^{(2-)}$.
\end{itemize}

Since $I(g \cdot f)=(\det g)^4I(f)$ and $J(g \cdot f)=(\det g)^6J(f)$,
it follows that two forms $f_1,f_2\in V_\R^{(i)}$ are
$\GL_2(\R)$-equivalent if and only if there exists a positive constant
$\lambda \in \R$ with $I(f_1)=\lambda^2I(f_2)$ and
$J(f_1)=\lambda^3J(f_2)$. Given a pair
$(I,J)\neq(0,0)$, there always exists a positive constant $\lambda$
such that $H(\lambda^2I,\lambda^3J)=1$.
Therefore, for $i=0$, $2+$, or $2-$ (resp.\ for $i=1$), a fundamental
set $L^{(i)}$ for the action of $\GL_2(\R)$ on $V^{(i)}_\R$ can be
constructed by choosing one form $f\in V^{(i)}_\R$, having invariants
$I$ and $J$, for each $(I,J)$ such that $H(I,J)=1$ and $4I^3-J^2>0$
(resp.\ $4I^3-J^2<0$). Table~\ref{table1} provides explicit
constructions of such fundamental sets $L^{(i)}$.  

\vspace{-.15in}
\begin{table}[ht]
\centering
\begin{eqnarray*}
L^{(0)}&=&\Bigl\{x^3y-\frac{1}{3}xy^3-\frac{J}{27}y^4:-2 <J <2\Bigr\}\\[.125in]
L^{(1)}&=&\Bigl\{x^3y-\frac{I}{3}xy^3+\frac{\pm 2}{27}y^4:-1\leq I < 1\Bigr\}\cup\Bigl\{x^3y+\frac{1}{3}xy^3-\frac{J}{27}y^4:-2 <J < 2\Bigr\}\\[.125in]
L^{(2+)}\!\!\!&=&\Bigl\{\frac{1}{16}x^4-\frac{\sqrt{2-J}}{3\sqrt{3}}x^3y+\frac{1}{2}x^2y^2+y^4:
-2<J< 2\Bigr\}\\[.125in]
L^{(2-)}\!\!\!&=&\left\{f:-f \in L^{2+}\right\}
\end{eqnarray*}\vspace{-.25in}
\caption{Explicit constructions of fundamental sets $L^{(i)}$ for
  $\GL_2(\R)\backslash V_\R^{(i)}$}\label{table1}
\end{table}
\vspace{-.075in}

\noindent
The key fact that
we use about these chosen fundamental sets $L^{(i)}$ is that the
coefficients of all the binary quartic forms in these $L^{(i)}$ are
 bounded; i.e., the $L^{(i)}$ all lie in a bounded
subset of $V_\R$.  It follows that, for any $h$ lying in a fixed compact
subset $G_0\subset \GL_2(\R)$, the set $h\cdot L^{(i)}$ is also a
fundamental set for the action of $\GL_2(\R)$ on $V_\R^{(i)}$, and
all coefficients are then bounded independent of $h$.

We will have need for the following lemma, whose proof is
postponed to \S\ref{secaux}:
\begin{lemma}\label{rstab}
  Let $f$ be an element in $V_\R^{(i)}$ having nonzero discriminant.
  Then the order of the stabilizer of $f$ in $\GL_2(\R)$ is $8$ if $i=0$ or
  $2$, and $4$ if $i=1$.
\end{lemma}

\vspace{.1in}

Let $\FF$ denote
Gauss's usual fundamental domain for $\GL_2(\Z)\backslash \GL_2(\R)$
in $\GL_2(\R)$. It follows from
\cite[Ch.\ 7, Th.\ 1]{Sercia} that $\FF$ may be expressed in the form $\FF=
\{n\alpha k\lambda:n(u)\in N'(t),\alpha(t)\in A',k\in K,\lambda\in\Lambda\},$ where
\begin{equation}\label{nak}
N'(t)= \left\{\left(\begin{array}{cc} 1 & {} \\ {u} & 1 \end{array}\right):
        u\in\nu(t) \right\}    , \;\;
A' = \left\{\left(\begin{array}{cc} t^{-1} & {} \\ {} & t \end{array}\right):
       t\geq \sqrt[4]3/\sqrt2 \right\}, \;\;
\Lambda = \left\{\left(\begin{array}{cc} \lambda & {} \\ {} & \lambda 
        \end{array}\right):
        \lambda>0 \right\},
\end{equation}
and $K$ is as usual the (compact) real orthogonal group ${\rm SO}_2(\R)$;
here $\nu(t)$ is a union of one or two subintervals of
$[-\frac12,\frac12]$ depending only on the value of $t$.

For $i=0$, $1$, $2+$, and $2-$, let $2n_i$ denote the cardinality of
the stabilizer in $\GL_2(\R)$ of an irreducible element $v\in
V^{(i)}_\R$.  Then, by Lemma \ref{rstab}, we have $n_0=4,$ $n_1=2,$
$n_{2+}=4,$ and $n_{2-}=4$. For $h\in\GL_2(\R)$, we regard $\FF h\cdot
L^{(i)}$ as a multiset, where the multiplicity of a point $x$ in
$\FF h\cdot L^{(i)}$ is given by the cardinality of the set
$\{g\in\FF\,\,:\,\,x\in gh\cdot L^{(i)}\}$.  
We claim that the $\GL_2(\Z)$-equivalence class of $x$ in $V^{(i)}_\R$ is
represented $m(x)
:=\#\Stab_{\GL_2(\R)}(x)/\#\Stab_{\GL_2(\Z)}(x)$ 
times in the multiset $\FF h\cdot L^{(i)}$; \;i.e., the multiplicity of 
$x'$ in $\FF h\cdot L^{(i)}$, summed over all $x'\in V_\Z$ that are
$\GL_2(\Z)$-equivalent to $x$, is equal to $m(x)$.  
Indeed, for any element $x\in
V_\R^{(i)}$, there exists a unique element $x_L \in h\cdot L^{(i)}$
that is $\GL_2(\R)$-equivalent to $x$. Suppose $g\in\GL_2(\R)$
satisfies $g\cdot x_L=x$.  Then for an element $g'\in\GL_2(\R)$, the
element $g'\cdot x_L\in V_\Z$ is $\GL_2(\Z)$-equivalent to~$x$ if and
only if $g'=\gamma g g_0$ for some $\gamma\in \GL_2(\Z)$ and
$g_0\in\Stab_{\GL_2(\R)}(x_L)$, i.e., if and only if
$g$ and~$g'$ map to the same element in the double coset
space $$\GL_2(\Z)\backslash\GL_2(\R)/\Stab_{\GL_2(\R)}(x_L).$$ The
number of such double cosets in the single right coset $\GL_2(\Z)g$ is
equal to
\begin{equation}
\frac{\#[g\Stab_{\GL_2(\R)}(x_L)g^{-1}]}{\#[\GL_2(\Z)\cap g\Stab_{\GL_2(\R)}(x_L)g^{-1}]}=
\frac{\#\Stab_{\GL_2(\R)}(x)}{\#\Stab_{\GL_2(\Z)}(x)}=m(x)
\end{equation}
as desired.

Since the stabilizer in $\GL_2(\Z)$ of an element $x\in V_\R$ always
contains the identity and its negative, $m(x)$ is always a number
between 1 and $n_i$.  In fact, for almost all $x\in V^{(i)}_\R$, the
quantity $m(x)$ is equal to~$n_i$.  Indeed, for any fixed
$\gamma\in\GL_2(\Z)$ not equal to plus or minus the identity, the set
of elements in $V_\R$ that are fixed by $\gamma$ has measure
$0$. Since $\GL_2(\Z)$ is countable, it follows that the set of
elements $x\in V^{(i)}_\R$ such that $m(x)<n_i$ also has
measure~$0$. Thus for any $h\in \GL_2(\R)$, away from a measure zero
set, the multiset $\FF h\cdot L^{(i)}$ is the union of $n_i$
fundamental domains for the action of $\GL_2(\Z)$ on $V^{(i)}_\R$.

Therefore, for any $h\in \GL_2(\R)$, if we let $\RR_X(h\cdot
L^{(i)})$ denote the multiset $\{w\in \FF h\cdot
L^{(i)}:|H(w)|<X\}$, then the product $n_iN(V_\Z^{(i)};X)$ is
equal to the number of irreducible integral points in $\RR_X(
h\cdot L^{(i)})$, with the slight caveat
that the (relatively rare---see Lemma~\ref{gl2zbigstab}) points with
$\GL_2(\Z)$-stabilizers of cardinality $2r$ ($r>1$) are counted with
weight $1/r$. 

As mentioned earlier, the main obstacle to counting integral points in
this region $\RR_X( h\cdot L^{(i)})$ is that it is not bounded,
but rather has a cusp going off to infinity (namely, the part of
$\RR_X( h\cdot L^{(i)})$ where the first
coordinate $a$ becomes small in absolute value, or equivalently, where the
parameter $t$ in (\ref{nak}) becomes large).  
We simplify the counting in this cuspidal region by  
``thickening'' the cusp; more precisely, we compute the number of
integral points in the region $\RR_X( h\cdot L^{(i)})$
by averaging over a ``compact continuum'' of such fundamental regions,
i.e., by averaging over the domains $\RR_X( h\cdot L^{(i)})$ where $h$
ranges over a certain compact subset $G_0\subset \GL_2(\R)$.  This
refinement of the method of \cite{dodpf} is described in more detail
in \S\ref{avgsec}.

However, we first turn in \S\ref{redsec} to bounding the
number of reducible points in the main bodies (i.e., away from the
cusps) of our fundamental regions.

\subsection{Estimates on reducibility}\label{redsec}

We consider the integral elements in the multiset $\RR_X(h\cdot
L^{(i)}):=\{w\in \FF h\cdot L^{(i)}:|H(w)|<X\}$ that are reducible
over $\Q$, where $h$ is any element in a fixed compact subset $G_0$ of
$\GL_2(\R)$.  Note that if a binary quartic form
$ax^4+bx^3y+cx^2y^2+dxy^3+ey^4$ satisfies $a=0$ (so that, in
particular, it lies in the cusp of the region $\RR_X(h\cdot
L^{(i)})$), then it is automatically reducible over $\Q$, since $y$
is a factor.  The following lemma shows that for integral binary
quartic forms in $\RR_X(h\cdot L^{(i)})$, reducibility with $a\neq0$
does not occur very often
(i.e., there are a negligible number of reducible points
in the main body of the fundamental domain):

\begin{lemma}\label{reducible}
  Let $h\in G_0$ be any element, where $G_0$ is any fixed compact
  subset of $\GL_2(\R)$.  Then the number of integral binary quartic
  forms $ax^4+bx^3y+cy^2+dxy^3+ey^4\in \RR_X(h\cdot L^{(i)})$ that are
  reducible over $\Q$ with $a\neq 0$ is $O(X^{2/3+\epsilon})$, where
  the implied constant depends only on $G_0$ and $\epsilon$.
\end{lemma}


\begin{proof}
  Let $f(x,y)=ax^4+bx^3y+cx^2y^2+dxy^3+ey^4$ be any element in
  $\RR_X(h\cdot L^{(i)})$. 
  We know that $\RR_X(h\cdot L^{(i)})\subset N'A'K\Lambda h\cdot L^{(i)}$,
  where $h\cdot L^{(i)}$ lies in a fixed compact
  set and $0<\lambda<X^{1/24}$.
Since all the coefficients of all the elements in
  $K\Lambda h\cdot L^{(i)}$ are bounded by $O((X^{1/24})^4)=O(X^{1/6})$,
  it follows that in  $N'A'K\Lambda h\cdot L^{(i)}$, we still
  have $a=O(X^{1/6})$, $b=O(X^{1/6})$, $c=O(X^{1/6})$,
  $ad=O(X^{2/6})$, $bd=O(X^{2/6})$, and $ae=O(X^{2/6})$.
In particular, the latter estimates clearly imply that the number of points in
  $\RR_X(h\cdot L^{(i)})$ with $a\neq 0$ and $e=0$ is $O(X^{4/6+\epsilon})$.

  Let us now assume that $a \neq 0$ and $e \neq 0$.  We first estimate
  the number of forms that have a rational linear factor.  The above
  estimates show that the number of possibilities for the quadruple
  $(a,b,d,e)$ is at most $O(X^{4/6+\epsilon})$.  If $px+qy$ is a
  linear factor of $f(x,y)$, where $p,q\in\Z$ are relatively prime,
  then $p$ must be a factor of $a$, while $q$ must be a factor of $e$;
  they are thus both determined up to $O(X^\epsilon)$ possibilities.
  Once $p$ and $q$ are determined, computing $f(-q,p)$ and setting it
  equal to zero then uniquely determines $c$ (if it is an integer at
  all) in terms of $a,b,d,e,p,q$.  Thus the total number of forms
  $f\in \RR_X(h\cdot L^{(i)})$ having a rational linear factor and
  $a\neq 0$ is $O(X^{4/6+\epsilon})$.

  We now estimate the number of binary quartic forms in
  $\RR_X(h\cdot L^{(i)})$ that factor into two irreducible binary
  quadratic forms over $\Z$, say
  $$ax^4+bx^3y+cx^2y^2+dxy^3+ey^4=(px^2+qxy+ry^2)
  \Bigl(\frac{a}{p}x^2+sxy+\frac{e}{r}y^2\Bigr)$$ where $p,q,r,s\in\Z$ and
  $p,q,r$ are relatively prime.  Since $ae=O(X^{2/6})$ and
  $a,e\neq 0$, the number of possibilities for the pair $(a,e)$ is
  $O(X^{2/6+\epsilon})$.  We then see that $p$ divides $a$ and $r$
  divides $e$, and hence the number of possibilities for $(p,r)$, once
  $a$ and $e$ have been fixed, is bounded by $O(X^\epsilon)$.

Next, equating coefficients, we see that:
\begin{equation}
\begin{array}{rcl}
\label{eq11}\displaystyle{\frac{a}{p}\:\!q\,+\,p\:\!s}&\!\!=\!\!&b,\\[.15in]
\displaystyle{\frac{e}{r}\:\!q\,+\,r\:\!s}&\!\!=\!\!&d.
\end{array}
\end{equation}

We split into two cases.  We first consider the case where
$\frac{ar}{pe}\neq \frac{p}{r}$, i.e., the linear system (\ref{eq11})
in the variables $q$ and $s$ is nonsingular.  Then the values of $b$
and $d$ uniquely determine $q$ and $s$, and so the total number of
quadruples $(a,b,d,e)$---and hence the total number of octuples $(a,b,d,e,p,r,q,s)$---is
at most $O(X^{4/6+\epsilon})$.  Furthermore, once this octuple has
been fixed, this also then determines $c$ by equating coefficients of
$x^2y^2$.  Hence there are at most $O(X^{4/6+\epsilon})$ possibilities
for $(a,b,c,d,e)$ in this case.

Next, we consider the case where $\frac{ar}{pe}=\frac{p}{r}$, so that
the system (\ref{eq11}) is singular.  In this case, the value of $b$
determines the value of $d$ uniquely, namely $d=(r/p)b$.  We have
already seen that there are $O(X^{2/6+\epsilon})$ possibilities for
the quadruple $(a,e,p,r)$.  Since there are only $O(X^{1/6})$ choices
for each of $b$ and $c$, and then $d$ is determined by $b$, the total
number of choices for $(a,b,c,d,e)$ is again $O(X^{4/6+\epsilon})$, as
desired.
\end{proof}

We also have the following lemma which bounds the number of
$\GL_2(\Z)$-equivalence classes of integral binary quartic forms having
large stabilizers inside $\GL_2(\Z)$ (in fact, in $\GL_2(\Q)$); we 
defer the proof to \S\ref{secaux}.

\begin{lemma}\label{gl2zbigstab}
  The number of $\GL_2(\Z)$-orbits of integral binary quartic forms
  $f\in V_\Z$ such that $\Delta(f)\neq 0$ and $H(f)<X$ whose stabilizer in $\GL_2(\Q)$ has size greater than
  $2$ is $O(X^{3/4+\epsilon})$.
\end{lemma}

\subsection{Averaging and cutting off the cusp}\label{avgsec}

Let $G_0$ be a compact, semialgebraic, left $K$-invariant set in
$\GL_2(\R)$ that is the closure of a nonempty open set and in which
every element has determinant greater than or equal to $1$. Then for
$i=0$, $1$, $2+$, and $2-$, we may write

\begin{equation}\label{equationfirst}
N(V_\Z^{(i)};X)=\frac{\int_{h\in G_0}\#\{x\in \FF h\cdot L\cap V_\Z^{\irr}:
  H(x)<X\}dh\;}{n_i\int_{h\in G_0}dh},
\end{equation}
where $V_\Z^\irr$ denotes the set of irreducible elements in
$V_\Z$, the set $L$ is equal to $L^{(i)}$, and $dh$ denotes
Haar-measure on $\GL_2(\R)$. We normalize $dh$ as follows: if we write
$h\in\GL_2(\R)$ in its Iwasawa decomposition as
$h=n(u)\alpha(t)k\lambda$, then $dh=t^{-2}du\,d^\times
t\,dk\,d^\times\lambda$, where $d^\times t=t^{-1}dt$, \;
$d^\times\lambda=\lambda^{-1}d\lambda$, \,and $\int_K dk=1$.
Thus, the denominator of the
right hand side of (\ref{equationfirst}) is an absolute constant
$C_{G_0}^{(i)}$ greater than zero.

More generally, for any $\GL_2(\Z)$-invariant subset $S \subset
V_\Z^{(i)}$, let $N(S;X)$ denote the number of irreducible
$\GL_2(\Z)$-orbits in $S$ having height less than $X$. Let $S^{\irr}$
denote the subset of irreducible points of $S$. Then $N(S;X)$ can be
similarly expressed as

\begin{equation}\label{eq9}
N(S;X)=\frac{\int_{h\in G_0}\#\{x\in \FF h\cdot L\cap S^{\irr}: H(x)<X\}dh\;}{C_{G_0}^{(i)}}.
\end{equation}
We use (\ref{eq9}) to define $N(S;X)$ even for
sets $S\subset V_\Z$ that are not necessarily $\GL_2(\Z)$-invariant.

Now, given $x\in V_\R^{(i)}$, let $x_L$ denote the {unique} point in $L$
that is $\GL_2(\R)$-equivalent to $x$. We have

\begin{equation}
N(S;X)=\frac{1}{C_{G_0}^{(i)}}\sum_{\substack{{x\in
      S^{\irr}}\\[.02in]{H(x)<X}}}\int_{h\in G_0} \#\{g \in \FF :
x=gh\cdot x_L\} dh.
\end{equation}
For a given $x\in S^{\irr}$, there exist a finite number of elements
$g_1,\ldots,g_n\in\GL_2(\R)$ satisfying $g_j\cdot x_L=x$.  We then have
$$\int_{h\in G_0} \#\{g \in \FF :x=gh\cdot x_L\} dh=\displaystyle\sum_j\int_{h\in G_0} \#\{g \in \FF :gh=g_j\} dh=\displaystyle\sum_j\int_{h\in G_0\cap\FF^{-1}g_j}dh.$$
As $dh$ is an invariant measure on $G$, we have
$$\displaystyle\sum_j\int_{h\in G_0\cap\FF^{-1}g_j}\!\!\!\!\!dh=\displaystyle\sum_j\int_{g\in G_0g_j^{-1}\cap\FF^{-1}}\!\!\!\!\!dg=\displaystyle\sum_j\int_{g\in\FF}\#\{h \in G_0 :gh=g_j\} dg=\int_{g\in \FF} \#\{h \in G_0 :x=gh\cdot x_L\} dg.$$
Therefore,
\begin{eqnarray}\label{eqavg}
N(S;X)&=&\frac{1}{C_{G_0}^{(i)}}\,\sum_{\substack{x\in S^{\irr}\\[.02in]
    H(x)<X}} \int_{g\in\FF} \#\{h \in G_0 : x=gh\cdot x_L\}dg\\
&\!\!=\!\!&  \frac1{C_{G_0}^{(i)}}\int_{g\in\FF}
\#\{x\in S^\irr\cap gG_0\cdot L:H(x)<X\}\,dg\\[.075in] 
&\!\!=\!\!&  \frac1{C_{G_0}^{(i)}}\int_{g\in N'(t)A'\Lambda K}
\#\{x\in S^\irr\cap n \bigl(\begin{smallmatrix}t^{-1}& {}\\
  {}& t\end{smallmatrix}\bigr) \lambda k G_0\cdot L:H(x)<X\}t^{-2} 
dn\, d^\times t\,d^\times \lambda\, dk\,.
\end{eqnarray}
Since $KG_0=G_0$ and $\int_K dk =1 $, we obtain the following theorem which provides a key formula
for $N(S,X)$:
\begin{theorem}
For any subset $S\subset V_\Z^{(i)}$, we have
\begin{equation}\label{avg}
N(S;X) = \frac1{C_{G_0}^{(i)}}\int_{g\in N'(t)A'\Lambda}                              
\#\{x\in S^\irr\cap B(n,t,\lambda,X)\}t^{-2}
dn\, d^\times t\,d^\times \lambda\,,
\end{equation}
where $C_{G_0}^{(i)}=n_i\int_{h\in G_0}dh$ and \begin{equation}\label{Bdef}B(n,t,\lambda,X) := n \bigl(\begin{smallmatrix}t^{-1}& {}\\ 
{}& t\end{smallmatrix}\bigr) \lambda G_0\cdot L\cap\{x\in V_\R^{(i)}:H(x)<X\}.\end{equation}
\end{theorem}

To estimate the number of lattice points in the (bounded) region $B(n,t,\lambda,X)$ defined by (\ref{Bdef}), we
have the following proposition due to
Davenport~\cite{Davenport1}. 

\begin{proposition}\label{davlem}
  Let $\mathcal R$ be a bounded, semialgebraic multiset in $\R^n$
  having maximum multiplicity $m$, and that is defined by at most $k$
  polynomial inequalities each having degree at most $\ell$.  
  Then the number of integral lattice points $($counted with
  multiplicity$)$ contained in the region $\mathcal R$ is
\[\Vol(\mathcal R)+ O(\max\{\Vol(\bar{\mathcal R}),1\}),\]
where $\Vol(\bar{\mathcal R})$ denotes the greatest $d$-dimensional 
volume of any projection of $\mathcal R$ onto a coordinate subspace
obtained by equating $n-d$ coordinates to zero, where 
$d$ takes all values from
$1$ to $n-1$.  The implied constant in the second summand depends
only on $n$, $m$, $k$, and $\ell$.
\end{proposition}
Davenport states the above proposition only for the number of lattice
points in compact semialgebraic sets $\mathcal
R\subset\R^n$. However, his result immediately implies
Proposition~\ref{davlem} for a general bounded semialgebraic multiset
$\mathcal R\subset\R^n$, via partitioning the multiset $\mathcal R$
into semialgebraic sets having constant multiplicity and then applying
the result to the closure and boundary of each such set.

By our construction of the $L^{(i)}$, the coefficients of the binary
quartic forms in $G_0\cdot L$ are all uniformly bounded.
Let $C$ be a constant such that $C^4$ bounds the absolute values of all the
coefficients of all the forms in $G_0\cdot L$.
We then have the following lemma on the number of lattice
points in $B(n,t,\lambda,X)$ having nonzero leading coefficient:

\begin{proposition}\label{aneq0}
The number of lattice points $(a,b,c,d,e)$ in $B(n,t,\lambda,X)$ with $a\neq0$ is
$$\left\{
\begin{array}{cl}
0 & \mbox{{\em if} ${C\lambda}<t$};\\[.1in]
\Vol(B(n,t,\lambda,X)) + O(t^4\lambda^{16})&\mbox{{\em otherwise.}}
\end{array}\right.$$
\end{proposition}

\begin{proof}
  If $ax^4+bx^3y+cx^2y^2+dxy^3+ey^4\in B(n,t,\lambda,X)$ is a binary
  quartic form, then $|a|$, $|b|$, $|c|$, $|d|$, and $|e|$ are at
  most $C^4\lambda^4/t^4$, $C^4\lambda^4/t^2$, $C^4\lambda^4$,
  $C^4\lambda^4t^2$, and $C^4\lambda^4t^4$, respectively.  If ${C
    {\lambda/ t}<1}$, then $a=0$ is the only possibility for such an
  integral binary quartic form. 

  Now assume $C\lambda/t\geq 1$.  This implies that $\lambda$, like
  $t$, is bounded below by a positive constant. Then each of 
the upper limits $C^4\lambda^4/t^4$, $C^4\lambda^4/t^2$, $C^4\lambda^4$,
  $C^4\lambda^4t^2$, and $C^4\lambda^4t^4$ for $|a|$,  $|b|$, $|c|$, $|d|$, and $|e|$, respectively, are also
bounded below by a positive constant, and the upper limit for $|a|$ is the smallest
of these upper limits up to a bounded constant.  Therefore, the
$k$-dimensional volume of any projection 
of $B(n,t,\lambda,X)$ onto a subspace defined by setting $k$
coefficients equal to 0 (where $1\leq k\leq 4$) is at most a bounded
constant times the product of the last four upper limits, or
  $O(\lambda^4/t^2\cdot\lambda^4\cdot\lambda^4t^2\cdot\lambda^4t^4)=O(t^4\lambda^{16})$.
  The result now follows from Proposition \ref{davlem}.
\end{proof}

In $(\ref{avg})$, since $L$ (and therefore also $G_0\cdot L$) contains only
points with height at least $1$, we observe (by the definition of
$B(n,t,\lambda,X)$) that the integrand will be nonzero only if
$t\leq C\lambda$ and $\lambda<X^{1/24}$.  Thus we may write 
\begin{equation}\label{eq15}
  N(V^{(i)}_\Z;X)=\frac1{C_{G_0}^{(i)}}
  \int_{\lambda=\sqrt[4]3/(\sqrt2C)}^{X^{1/24}}
  \int_{t=\sqrt[4]3/\sqrt2}^{C\lambda} \int_{N'(t)}
  (\Vol(B(n,t,\lambda,X)) +
  O({t^4}{\lambda^{16}})) t^{-2}
  dn\,d^\times t\,d^\times \lambda+O(X^{3/4+\epsilon}),
\end{equation}
where the error term of $O(X^{3/4+\epsilon})$ arises due to the bound
on reducible forms in Lemma~\ref{reducible} and the bound on forms
having nontrivial $\GL_2(\Z)$-stabilizer in Lemma~\ref{gl2zbigstab}.
The integral of the second summand is immediately evaluated to be
$O(X^{3/4})$.  Meanwhile, the integral of the first summand is
\begin{equation}\label{eq16}
\frac1{C_{G_0}^{(i)}}\int_{h\in G_0}{\rm Vol}
(\mathcal R_X(h\cdot L)){dh} -\int_{\lambda=\sqrt[4]3/(\sqrt2C)}^{X^{1/24}}
\int_{t=C\lambda}^{\infty}\int_{N'(t)}
\Vol(B(n,t,\lambda,X))t^{-2} 
dn \,d^\times t \,d^\times \lambda.
\end{equation}
However, ${\rm Vol}(\mathcal R_X(h\cdot L))$ is independent of $h$;
also, since ${\rm Vol}(B(n,t,\lambda,X))=O(\lambda^{20})$, by carrying
out the integration in the second term of (\ref{eq16}), we see that
that this term is also $O(X^{3/4})$. In other words, the volume of the
cuspidal region, where $t>C\lambda$, is small. We conclude that
\begin{equation}\label{bigint}
N(V^{(i)}_\Z;X)={\rm Vol}(\mathcal R_X(L))/n_i+O(X^{3/4+\epsilon}).
\end{equation}
To complete the proof of Theorem \ref{refbq}, it thus remains only to
compute the volume ${\rm Vol}(\mathcal R_X(L))$.

\subsection{Computation of the volume}\label{secvol}

Let $i$ be equal to $0$, $1$, $2+$, or $2-$.  Our aim in this subsection is to
compute the volume of $\mathcal R_X(L^{(i)})=\{w\in \FF h\cdot L^{(i)}:|H(w)|<X\}$.  To this end, let
$R^{(i)}:=\Lambda\cdot L^{(i)}$. Then for each $(I,J)\in
\R\times\R$ with $\Delta(I,J)>0$, the sets $R^{(0)}$, $R^{(2+)}$,
and $R^{(2-)}$ contain exactly one point having invariants $I$ and
$J$; for each $(I,J)\in\R\times\R$ with $\Delta(I,J)<0$, the set
$R^{(1)}$ contains exactly one point having invariants $I$ and $J$.
Let $R^{(i)}(X)$ denote the set of
all those points in $R^{(i)}$ having height less than $X$. 
We now consider a twisted action of $\GL_2(\R)$ on $V_\R$ given by
\begin{equation}
\gamma\cdot
f(x,y):=f((x,y)\cdot\gamma)/(\det \gamma)^2
\end{equation}
for $\gamma\in\GL_2(\R)$ and $f\in V_\R$, which induces an action of
$\PGL_2(\R)$ on $V_\R$.  
Let $\FF_{\PGL_2}$ be the
image in $\PGL_2(\R)$ of the fundamental domain $\FF$ for the action
of $\GL_2(\Z)$ on $\GL_2(\R)$. Then $\FF_{\PGL_2}$ is a fundamental
domain for the action of $\PGL_2(\Z)$ on $\PGL_2(\R)$ by left multiplication.
Furthermore, we have $\mathcal R_X(L^{(i)})=\FF_{\PGL_2}\cdot R^{(i)}(X)$.

The set $R^{(i)}$ is in canonical one-to-one correspondence with the
set $\{(I,J)\in \R\times\R:I^3-J^2/4>0\}$ if $i=0$, $2+$, or $2-$, and
with $\{(I,J)\in \R\times\R:I^3-J^2/4<0\}$ if $i=1$. There is thus a
natural measure on each of these sets $R^{(i)}$, given by
$dr=dI\,dJ$. Let $\omega$ be a differential which generates the rank
$1$ module of top-degree differentials of $\PGL_2$ over $\Z$. Then
$\omega$ is well-defined up to sign. 
To compute the volume of the multiset $\mathcal R_X(L^{(i)})=\FF_{\PGL_2}\cdot R^{(i)}(X)$, we use the following proposition:
\begin{proposition}\label{bqjac}
For any measurable function $\phi$
  on $V_\R$, we have
\begin{equation}\label{Jac}
\int_{\FF_{\PGL_2}\cdot R^{(i)}}\phi(v)dv=\frac{1}{27}\int_{R^{(i)}}
\int_{\PGL_2(\R)}\phi(g\cdot p^{(i)}_{I,J})\,\omega(g)\,dI dJ,
\end{equation}
where $p^{(i)}_{I,J}\in R^{(i)}$ is the point having invariants equal to $I$ and $J$ and we regard $\FF_{\PGL_2}\cdot R^{(i)}$ as a multiset.
\end{proposition}
The proposition follows from a
Jacobian computation and can be verified directly; for a more noncomputational proof of the above proposition, see Section 3.3.

Proposition \ref{bqjac} may now be used
to compute the volume of the multiset $\mathcal R_X(L^{(i)})$; we have
\begin{equation}
\int_{\mathcal R_X(L^{(i)})}\!\!\!\!\!dv=\int_{\FF_{\PGL_2}\cdot R^{(i)}(X)}\!\!\!\!\!dv=
\frac{1}{27}\int_{R^{(i)}(X)}\int_{\FF_{\PGL_2}}dg\,dI\,dJ=\frac{2\zeta(2)}{27}\int_{R^{(i)}(X)}dI\,dJ,
\end{equation}
where the final equality follows from the fact that $\Vol(\FF_{\PGL_2})=\Vol(\PGL_2(\Z)\backslash\PGL_2(\R))=2\zeta(2)$ (see \cite{Langlands}).
When $i=0$, $2+$, or $2-$, we compute $\int_{R^{(i)}(X)}dI\,dJ$ to be
\begin{equation}\label{volrplus}
\int_{I=0}^{X^{1/3}}\int_{J=-2I^{3/2}}^{2I^{3/2}}dJdI=\frac{8}{5}X^{5/6}.
\end{equation}
Meanwhile, $\int_{R^{(1)}(X)}dI\,dJ$ is equal to
\begin{equation}\label{volrminus}
\int_{I=-X^{1/3}}^{X^{1/3}}\int_{J=-2X^{1/2}}^{2X^{1/2}}dJdI-\Vol(R^{(0)}(X))=8X^{5/6}-\frac{8}{5}X^{5/6}=\frac{32}{5}X^{5/6}.
\end{equation}
We conclude that 
\begin{equation}\label{eq22}
\Vol(\mathcal R_X(L^{(i)})) =
\left\{\begin{array}{ll}
\displaystyle{\frac{16}{135}\cdot\zeta(2)X^{5/6}}
&\qquad\mbox{ for $i=0$, $2+$, and $2-$; }\\[.1in]
\displaystyle{\frac{64}{135}\cdot\zeta(2)X^{5/6}}&\qquad\mbox{ for $i=1$. }
\end{array}\right.
\end{equation}
As $n_0=n_{2+}=n_{2-}=4$ and $n_1=2$, Equations (\ref{bigint}) and (\ref{eq22}) now
immediately imply Theorem~\ref{refbq}.

To deduce Theorem \ref{bqaverage} from Theorem \ref{refbq}, we require a
count of the number of eligible pairs $(I,J)\in\Z\times\Z$ satisfying
$H(I,J)<X$.  The next lemma follows immediately from Theorem \ref{eligible}, which we
prove in \S\ref{secaux}:
\begin{lemma}\label{conglat}
  The set of eligible $(I,J)\in\Z\times\Z$ is a union of $9$ distinct
  translates of $9\Z\times 27\Z$.
\end{lemma}

The following proposition is now a simple application of Proposition \ref{davlem} and Lemma \ref{conglat}.

\begin{proposition}\label{IJcount}
  Let $N^{+}_{I,J}(X)$ and $N^{-}_{I,J}(X)$ denote the number of
  eligible $(I,J)\in\Z\times\Z$ satisfying $H(I,J)<X$ that have
  positive discriminant and negative discriminant, respectively. Then
  we have
\begin{itemize}
\item[{\rm (a)}]$N^{+}_{I,J}(X)=\displaystyle\frac{8}{135}X^{5/6}+O(X^{1/2});$
\item[{\rm (b)}]$N^{-}_{I,J}(X)=\displaystyle\frac{32}{135}X^{5/6}+O(X^{1/2}).$
\end{itemize}
\end{proposition}
\begin{proof}
  Let $R_{I,J}^{\pm}(X)$ denote the sets
  $\{(i,j)\in\R^2:|i|<X^{1/3},\,|j|<2X^{1/2},\,\pm(4i^3-j^2)>0\}$. The
  sizes of the projections of $R_{I,J}^{\pm}(X)$ onto
  smaller-dimensional coordinate hyperplanes are all bounded by
  $O(X^{1/2})$. Using Proposition \ref{davlem} and Lemma~\ref{conglat}
  we then see that
  $N^{\pm}_{I,J}(X)=\frac{9}{243}\Vol(R_{I,J}^{\pm}(X))+O(X^{1/2})$.
  The volumes of $R_{I,J}^{+}(X)$ and $R_{I,J}^{-}(X)$ were computed
  in (\ref{volrplus}) and (\ref{volrminus}), respectively, and the
  proposition follows.
\end{proof}

Theorem \ref{bqaverage} now follows from Theorem \ref{refbq} and Proposition~\ref{IJcount}.

\subsection{Congruence conditions}\label{secbqcong}

In this subsection, we prove a version of Theorem \ref{refbq} where we
count integral binary quartic forms satisfying any specified finite
set of congruence conditions.

Suppose $S$ is a subset of $V_\Z$ defined by finitely many congruence
conditions. We may assume that $S\subset V_\Z$ is defined by
congruence conditions modulo some integer $m$.  Then $S$ may be viewed
as the union of (say) $k$ translates $\mathcal L_1,\ldots,\mathcal
L_k$ of the lattice $m\cdot V_\Z$. For each such lattice translate
$\mathcal L_j$, we may use formula (\ref{avg}) and the discussion
following that formula to compute $N(\mathcal L_j\cap V_\Z^{(i)};X)$,
where each $d$-dimensional volume is scaled by a factor of $1/m^d$ to
reflect the fact that our new lattice has been scaled by a factor of
$m$.  With these scalings, the maximum volume of the projections of
$B(n,t,\lambda,X)$ is seen to be at most $O(t^4\lambda^{16})$.
Analogous to Proposition~\ref{aneq0}, we see that the number of points
$(a,b,c,d,e)$ in $B(n,t,\lambda,X)\cap \mathcal L_j$ with $a\neq0$ is
$$\left\{
\begin{array}{cl}
0 & \mbox{{\rm if} $\frac{C\lambda}{t}<1$};\\[.1in]
\displaystyle\frac{1}{m^5}\Vol(B(n,t,\lambda,X)) + O({t^4}
{\lambda^{16}})&\mbox{{\rm otherwise.}}
\end{array}\right.$$

Carrying out the integral for $N(\mathcal L_j\cap V_\Z^{(i)};X)$ as in
(\ref{eq15})--(\ref{eq16}), we obtain
the following analogue of~(\ref{bigint}):
$$N(\mathcal L_j\cap V_\Z^{(i)};X) = \frac{\Vol(\RR_X(L^{(i)}))}{n_i\cdot m^5}+O(X^{3/4+\epsilon}).$$
Summing over $j$, we thus obtain
\begin{equation}\label{oo}
N(S\cap V_\Z^{(i)};X) = \frac{k\Vol(\RR_X(L^{(i)}))}{n_i\cdot m^5}
+O(X^{3/4+\epsilon}).
\end{equation}

For any set $S$ in $V_\Z$ that is
definable by congruence conditions, let us denote by $\mu_p(S)$
the $p$-adic density of the $p$-adic closure of $S$ in $V_{\Z_p}$,
where we normalize the additive measure $\mu_p$ on $V_{\Z_p}$ so that
$\mu_p(V_{\Z_p})=1$.
We then have the following theorem:
\begin{theorem}\label{cong2}
Suppose $S$ is a subset of $V_\Z$ defined by 
congruence conditions modulo finitely many prime powers. Then we have
\begin{equation}\label{ramanujan}
N(S\cap V_\Z^{(i)};X)
  = N(V_\Z^{(i)};X)
  \prod_{p} \mu_p(S)+O(X^{3/4+\epsilon}),
\end{equation}
where $\mu_p(S)$ denotes the $p$-adic density of $S$ in $V_\Z$, and
where the implied constant depends only on $S$ and $\epsilon$.
\end{theorem}
Theorem~\ref{cong2} follows from Equations (\ref{bigint}) and
(\ref{oo}), together with the identity $km^{-5}=\prod_p\mu_p(S)$.

We will also have occasion to use the following weighted version of
Theorem \ref{cong2}; the proof is identical.

\begin{theorem}\label{cong3}
  Let $p_1,\ldots,p_k$ be distinct prime numbers. For $j=1,\ldots,k$,
  let $\phi_{p_j}:V_\Z\to\R$ be a $\GL_2(\Z)$-invariant function on
  $V_\Z$ such that $\phi_{p_j}(f)$ depends only on the congruence
  class of $f$ modulo some power $p_j^{a_j}$ of $p_j$.  Let
  $N_\phi(V_\Z^{(i)};X)$ denote the number of irreducible
  $\GL_2(\Z)$-orbits in $V_\Z^{(i)}$ having height bounded by $X$,
  where each orbit $\GL_2(\Z)\cdot f$ is counted with weight
  $\phi(f):=\prod_{j=1}^k\phi_{p_j}(f)$. Then we have
\begin{equation}
N_\phi(V_\Z^{(i)};X)
  = N(V_\Z^{(i)};X)
  \prod_{j=1}^k \int_{f\in V_{\Z_{p_j}}}\tilde{\phi}_{p_j}(f)\,df+O(X^{3/4+\epsilon}),
\end{equation}
where $\tilde{\phi}_{p_j}$ is the natural extension of ${\phi}_{p_j}$
to $V_{\Z_{p_j}}$, $df$ denotes the additive
measure on $V_{\Z_{p_j}}$ normalized so that $\int_{f\in
  V_{\Z_{p_j}}}df=1$, and where the implied constant in the error term
depends only on the local weight functions ${\phi}_{p_j}$ and $\epsilon$.
\end{theorem}

\subsection{Uniformity estimates}\label{secunif}

In order to prove Theorems \ref{mainellip} and \ref{ellipcong}, we
require a sieve that allows us to count equivalence classes of
integral binary quartic forms of
bounded height satisfying certain infinite sets of congruence conditions. (In
particular, this sieve will allow us to count equivalences classes of
integral binary quartic forms having
bounded height and {\it squarefree} discriminant.)  A key ingredient for this
sieve---and the purpose of this subsection---is an estimate that
uniformly bounds the error terms in Theorems~\ref{cong2} and
\ref{cong3} as more and more congruence conditions are
imposed. 

Specifically, we prove the following theorem:

\begin{theorem}\label{thunifbqelem}
  For a prime $p$, let $\W_p(V)$ denote the set of binary
  quartic forms $f\in V_\Z$ such that $p^2\mid\Delta(f)$. Then, for any $M>0$, we have:
$$
\lim_{X\to\infty}\frac{N(\cup_{p>M}\W_p(V);X)}{X^{5/6}}=O\Bigl(\frac1{\log\,M}\Bigr),
$$
where the implied constant is independent of $M$.
\end{theorem}

Such uniformity estimates can in general be quite nontrivial.  In the
current case, to prove this estimate, we use the following
trick.  We embed the space of integral binary quartic forms into the
space of pairs of integral ternary quadratic forms, where such an
estimate has been proven previously~\cite[Proposition~23]{dodqf}.  
More precisely, let
$W_\Z$ denote the space of pairs $(A,B)$ of ternary quadratic forms
having coefficients in $\Z$. We will always identify ternary quadratic
forms over $\Z$ with their Gram matrices whose coefficients lie in
$\frac12\Z$; we may thus express an element $(A,B)\in W_\Z$ as a pair
of $3\times 3$ symmetric matrices via
$$2 \cdot (A,B)=\left(\left[\begin{array}{ccc} 2a_{11} & a_{12} & a_{13} \\ a_{12} & 2a_{22} & a_{23} \\ a_{13} & a_{23} & 2a_{33} \end{array} \right],
  \left[ \begin{array}{ccc} 2b_{11} & b_{12} & b_{13} \\ b_{12} &
      2b_{22} & b_{23} \\ b_{13} & b_{23} & 2b_{33} \end{array}
  \right] \right),$$ where $a_{ij},b_{ij}\in \Z$. 

The group
$\GL_2(\Z)\times \SL_3(\Z)$ acts naturally on the space $W_\Z$.
Namely, an element $g_3\in \SL_3(\Z)$ acts on $W_\Z$ by
$g_3\cdot(A,B)=(g_3Ag_3^t,g_3Bg_3^t)$, while an element
$g_2=\left(\begin{smallmatrix}{} p & q\\r & s \end{smallmatrix}\right)
\in \GL_2(\Z)$ acts by $g_2\cdot(A,B)=(pA+qB,rA+sB)$. The ring of
polynomial invariants for the action of $\GL_2(\Z)\times\SL_3(\Z)$ on
$W_\Z$ is generated by one element, which is called the
{\it discriminant}. The discriminant $\Delta(A,B)$ of an element $(A,B)\in W_\Z$ is
given by the discriminant of the binary cubic form $4\,\Det(Ax-By)$ in $x$ and
$y$, and is thus an invariant of degree 12 in the entries of $A$ and $B$.

The space $V_\Z$ of integral binary quartic forms embeds into $W_\Z$
via the map $\phi$ defined by
\begin{equation}\label{vtow}
\phi: ax^4+bx^3y+cx^2y^2+dxy^3+ey^4\mapsto \left( \left[ \begin{array}{ccc}\phantom0 & \phantom0 & 1/2 \\ \phantom0 & -1 & \phantom0 \\ 1/2 & \phantom0 & \phantom0 \end{array} \right],
\left[ \begin{array}{ccc} a & b/2 & 0 \\ b/2 & c & d/2 \\ 0 & d/2 & e \end{array} \right]\right).
\end{equation}
We denote the first matrix in (\ref{vtow}) by $A_1$, and the subset
of all pairs $(A_1,B)$ of ternary quadratic forms in $W_\Z$ 
by $W_{\Z,1}$. The group
$F_{\Z,1}\times\SO(A_1)\subset\GL_2(\Z)\times\SL_3(\Z)$ preserves
$W_{\Z,1}$, where $F_{\Z,1}$ is the group of all $2\times 2$ lower
triangular matrices over $\Z$ with $1$'s on the diagonal. We also note that the
map $\phi$ is {\it discriminant preserving}, i.e., the discriminant of
an element of $V_\Z$ is equal to the discriminant of its image in
$W_\Z$.  For a binary quartic form $f$, if we write
$\phi(f)=(A_1,B)$, then we
call the binary form $\Det(Ax-By)$ the 
{\it cubic resolvent form} of $f$; note that this form is {\it monic},
i.e., its leading coefficient as a polynomial in $x$ is 1. 

Next, we observe that every 
$F_{\Z,1}$-equivalence class of $W_{\Z,1}$ contains a unique element
$(A_1,B)$ such that the top right entry of $B$ is equal to $0$.
It follows that $\phi$ maps the space of binary quartic forms $V_\Z$ bijectively to
the set of $F_{\Z,1}$-orbits on $W_{\Z,1}$ via the
composite map
$$V_\Z\rightarrow W_{\Z,1}\rightarrow F_{\Z,1}\backslash W_{\Z,1}.$$

We may ask how the action of $\GL_2(\Z)$ on $V_\Z$ manifests itself (via $\phi$) as an action on $F_{\Z,1}\backslash W_{\Z,1}$.
To answer this, note that the center of $\GL_2(\Z)$ acts trivially on its representation on
binary quadratic forms $px^2-2qxy+ry^2$ via $\gamma\cdot
f(x,y):=f((x,y)\cdot\gamma)/(\det \gamma)$.
This action of $\GL_2(\Z)$ preserves the discriminant $4(q^2-pr)$ of these binary quadratic forms, yielding the map
\begin{equation}\label{eqtwistedaction1}
\begin{array}{rcl}
\rho:\PGL_2(\Z)&\rightarrow&\SL_3(\Z),\;\; \mbox{given explicitly by}\\[.1in]
{\left(\begin{array}{cc} a & {b} \\ c & d \end{array}\right)}&\mapsto& \displaystyle\frac{{1}}{ad-bc}{\left(\begin{array}{ccc} {d^2} & {cd} & {c^2} \\ {2bd} & {ad+bc} & {2ac} \\ {b^2} & {ab} & {a^2}
\end{array}\right).}
\end{array}
\end{equation}
Since $A_1$ is the Gram matrix of the ternary form
$q^2-pr$, we see that
the image of $\PGL_2(\Z)$ is contained in the orthogonal group
$\SO(A_1,\Z)$, and is in fact equal to it (see \cite[Lemma~4.4.2]{MWt}).

For any ring $R$, let $V_R$ denote the space of binary quartic forms
with coefficients in $R$. The center of $\GL_2(R)$ acts trivially under
the ``twisted action'' of $\GL_2(R)$ on $V_R$ defined by 
\begin{equation}\label{eqtwistedaction}
\gamma\cdot f(x,y):=(\det\gamma)^{-2}f((x,y)\cdot\gamma),  
\end{equation}
yielding an action of $\PGL_2(R)$ on $V_R$. Note that the
$\PGL_2(\Z)$-orbits on $V_\Z$ are the same as the $\GL_2(\Z)$-orbits
on $V_\Z$, since
$\left(\begin{smallmatrix}{-1}&{}\\{}&{-1}\end{smallmatrix}\right)\in\GL_2(\Z)$
acts trivially on $V_\Z$.

It is now easily checked that $\phi(\gamma\cdot f)$ and
$\rho(\gamma)\cdot\phi(f)$ are the same element in $F_{\Z,1}\backslash
W_{\Z,1}$ for all $\gamma\in\PGL_2(\Z)$ and $f\in V_\Z$. Therefore,
we have the following theorem, which will be essential in proving 
the uniformity estimate of Theorem~\ref{thunifbqelem}:

\begin{theorem}\label{thmvztowz}
  The map $\phi$ defined by $(\ref{vtow})$ gives a canonical bijection
  between $\PGL_2(\Z)$-orbits on $V_\Z$ and $F_{\Z,1}\times {\rm
    SO}(A_1,\Z)$-orbits on $W_{\Z,1}$.
\end{theorem}

We thus obtain a natural map 
\begin{equation}\label{eqpsimap}
\psi:\PGL_2(\Z)\backslash V_\Z\to (\GL_2(\Z)\times\SL_3(\Z))\backslash W_\Z  
\end{equation}
given by the composite map
\begin{equation}\label{eqvtowcomp}
\PGL_2(\Z)\backslash V_\Z \to 
   (F_{\Z,1}\times {\rm SO}(A_1,\Z))\backslash 
W_{\Z,1}\to(\GL_2(\Z)\times\SL_3(\Z))\backslash W_\Z.
\end{equation}

\vspace{-0.05in}

\begin{remark}
{\em It is proven in \cite{quarparam} that the
orbit space $(\GL_2(\Z)\times\SL_3(\Z))\backslash W_\Z$ corresponds to
isomorphism classes of 
pairs $(Q,R)$, where $Q$ is a quartic ring and $R$ is a cubic resolvent ring of $Q$.
Meanwhile, using the map (\ref{vtow}), Wood \cite{Wood} proves
that the orbit space $\PGL_2(\Z)\backslash V_\Z$ corresponds to
isomorphism classes of triples $(Q,R,x)$, where $Q$ is a quartic ring, $R$ is a monogenic cubic resolvent ring 
of $Q$, and $x$ is a {\it monogenizer} of $R$, i.e., $x$ generates $R$ as a $\Z$-algebra (so that $R=\Z[x]$).
It follows that the map $\psi$ in (\ref{eqpsimap}) corresponds to the map
$$\{(Q,R,x)\} \to \{(Q,R)\},$$
which takes a quartic ring with a monogenized cubic resolvent ring and simply forgets its monogenizer (and the fact that $R$ is monogenic).}
\end{remark}

Before we state and prove the desired uniformity estimate, we require
the following key proposition:

\begin{proposition}\label{atmosttwelve}
  An element of $(\GL_2(\Z)\times\SL_3(\Z))\backslash W_\Z$ with nonzero discriminant has at most $12$ preimages in $\PGL_2(\Z)\backslash V_\Z$ under the map $\psi$.
\end{proposition}
\begin{proof}
  By Theorem \ref{thmvztowz}, it suffices to prove that an element $w$ of
  $\GL_2(\Z)\times\SL_3(\Z)\backslash W_\Z$ has at most $12$ preimages
  in $F_{\Z,1}\times\SO(A_1,\Z)\backslash W_{\Z,1}$.  Let 
  $\{(A_1,B_\alpha)\}$ be a set of $F_{\Z,1}\times\SO(A_1,\Z)$-inequivalent preimages of $w$ in 
  $W_{\Z,1}$, where $\alpha$ ranges over
  some (possibly infinite) set $\mathcal{A}$. The integral
  binary cubic forms $g_\alpha(x,y):=4\,\Det(A_1x-B_\alpha y)$ all 
have $x^3$-coefficient equal to $1$, i.e., $g_\alpha(1,0)=1$.
Since the  $(A_1,B_\alpha)$ are pairwise
  $F_{\Z,1}\times\SO(A_1,\Z)$-inequivalent but are all
  $\GL_2(\Z)\times\SL_3(\Z)$-equivalent, we see that the $g_\alpha$
  are pairwise $F_{\Z,1}$-inequivalent but are all $\GL_2(\Z)$-equivalent.
  
  The deep results in \cite{Del} and \cite{Ev}, which assert that
  $g(x,y)=1$ has at most $12$ solutions with $(x,y)\in\Z\times\Z$ for
  an integral binary cubic form $g$ of nonzero discriminant, now imply that the cardinality of
  $\mathcal{A}$ is at most $12$.
\end{proof}

We may now proceed to the proof of Theorem \ref{thunifbqelem}. To this
end, let $\W_p(V)\subset V_\Z$ denote the set of integral binary
quartic forms $f$ such that $p^2\mid\Delta(f)$.  We partition
$\W_p(V)$ into two disjoint sets $\W_p^{(1)}(V)$ and $\W_p^{(2)}(V)$.
Here, $\W_p^{(1)}(V)$ is the set of
all binary quartic forms $f$ whose discriminant is {\it strongly
  divisible} by $p^2$, i.e., $p^2\mid \Delta(f+pg)$ for all $g\in
V_\Z$. The set $\W_p^{(2)}(V)$ is the set of all binary quartic forms $f\in
V_\Z$ whose discriminant, in the terminology of \cite{geosieve}, is
{\it weakly divisible} by $p^2$, i.e., there exists $g\in V_\Z$ such
that $p^2\nmid \Delta(f+pg)$.

Then an element $f\in\W_p^{(1)}(V)$ is either a multiple of
$p$ or the {\it splitting type} of $f$ at $p$ is $(1^31)$, $(1^21^2)$,
$(2^2)$, or $(1^4)$, i.e., either $f\in pV_\Z$ or the reduction of $f$
modulo $p$ factors into irreducible factors over $\F_p$ as $c(x-\alpha
y)^3(x-\beta y)$, $c(x-\alpha y)^2(x-\beta y)^2$, $c(x^2+\alpha
xy+\beta y^2)^2$, or $c(x-\alpha y)^4$, respectively.

The desired uniformity estimate for $\W_p^{(1)}(V)$ follows by
applying the following quantitative version of a result of
Ekedahl~\cite{Ek}, proven in \cite[Theorem~3.3]{geosieve}:

\begin{theorem}\label{gsthm}
  Let $B$ be a compact region in $\R^n$ having finite measure, and let
  $Y$ be any closed subscheme of $\A^n_\Z$ of codimension $k\geq2$.
Let $r$ and $M$ be positive real numbers.  Then we have
\begin{equation}
\#\{v\in rB\cap \Z^n\,\,|\,\,v\!\!\!\!\!\pmod{p} \in Y(\F_p)\,\,\mbox{for some prime } p>M\} \,=\, 
O\left(\frac{r^n}{M^{k-1}\log M}+r^{n-k+1}\right),
\end{equation}
where the implied constant depends only on $B$ and on $Y$.
\end{theorem}
To apply this result, 
recall that we used $\FF_{\PGL_2}$ to denote the
fundamental domain $N'A'K$ for the left action of $\PGL_2(\Z)$ on
$\PGL_2(\R)$. For $0<\epsilon<1$, we denote by
$\FF^{(\epsilon)}_{\PGL_2}$ the subset of elements
$n(u)a(t)k\in\FF_{\PGL_2}$ where $t$ is bounded above by a suitable
constant to ensure that
$$\Vol(\FF^{(\epsilon)}_{\PGL_2})=(1-\epsilon)\Vol(\FF_{\PGL_2}).$$
Then,
for fixed $\epsilon>0$, the set $\FF^{(\epsilon)}_{\PGL_2}\cdot
R^{(i)}(X)$ (with $R^{(i)}(X)$ as defined in \S2.4) is a bounded
region in $V_\R$ that expands homogeneously as $X$ grows. We have
the following theorem:

\begin{theorem}\label{thunifbqelem1}
Let $0<\epsilon<1$ be fixed. For $i\in\{0,1,2+,2-\}$, we have
  \begin{equation}\label{equnifbqsm1}
    \#\left\{\FF^{(\epsilon)}_{\PGL_2}\cdot
      R^{(i)}(X)\bigcap(\displaystyle\cup_{p>M}\W^{(1)}_p(V))\right\}=O(X^{5/6}/(M\log
    M)+X^{2/3}),
  \end{equation}
where the implied constant depends only on $\epsilon$.
\end{theorem}
Indeed, the discriminants of elements in $\W_p^{(1)}(V)$ are strongly
divisible by $p^2$.  Theorem~\ref{thunifbqelem1} thus follows from Theorem~\ref{gsthm} (with $n=5$, $k=2$, and $r=X^{1/6}$) because, as noted in~\cite{geosieve},
if an element in $v\in V_\Z$ has discriminant $\Delta$ strongly
divisible by $p^2$, then it lies in $Y(\F_p)$, where $Y$ is the
codimension 2 subscheme of $V\cong \A^5$ defined by the vanishing of
$\Delta$ and $\partial\Delta/\partial e$.

However, a uniformity estimate for $\W_p^{(2)}(V)$---the set of
elements in $V_\Z$ having discriminant divisible, but not strongly
divisible, by $p^2$---is more difficult to obtain. It is for this case
that we consider the embedding (\ref{vtow}) of $V_\Z$ into $W_\Z$,
where we can then use previously obtained uniformity estimates for
$W_\Z$. We state the relevant estimate for $W_\Z$ below:
\begin{theorem}\label{unifw}
{\bf (\cite[Proposition 23]{dodqf})} Let $\W_p^{(2)}(W)$ denote the set of elements in $W_\Z$ whose
  discriminants are divisible, but not strongly divisible, by
  $p^2$. Then the number of $\GL_2(\Z)\times\SL_3(\Z)$-orbits on
  $\W_p^{(2)}(W)$ having discriminant bounded by $X$ is $O(X/p^2)$, where
  the implied constant is independent of $p$.
\end{theorem}
We may use this uniformity estimate for $\W_p^{(2)}(W)$ to obtain one
for $\W_p^{(2)}(V)$.  Specifically, in conjunction with Proposition
\ref{atmosttwelve}, we obtain the estimate
\begin{equation}\label{equnifbqfirst}
  N(\W_p^{(2)}(V);X)=O(X/p^2),
  \end{equation}
  where the implied constant is independent of $X$ and $p$. 

\begin{theorem}\label{thunifbqelem2}
Let $0<\epsilon<1$ be fixed. For $i\in\{0,1,2+,2-\}$, we have
  \begin{equation}\label{equnifbqsm2}
    \#\left\{\FF^{(\epsilon)}_{\PGL_2}\cdot R^{(i)}(X)\bigcap(\displaystyle\cup_{p>M}\W^{(2)}_p(V))\right\}=O(X^{5/6}/\log\,M),
  \end{equation}
where the implied constant is independent of $X$ and $M$.
\end{theorem}
\begin{proof}
  We define $R_X^{(\epsilon)}:=\FF^{(\epsilon)}_{\PGL_2}\cdot
  R^{(i)}(X)$ and obtain an individual bound on
  $\#\{R_X^{(\epsilon)}\cap \W^{(2)}_p(V)\}$ for each prime $p$. When
  viewed as a polynomial in $e$, the derivative of $\Delta$ with
  respect to $e$ is a nonzero cubic polynomial
  $\partial\Delta/\partial e$ in $e$.  If a binary quartic form
  $f(x,y)=a_0x^4+b_0x^3y+c_0x^2y^2+d_0xy^3+e_0y^4$ belongs to
  $\W_p^{(2)}$, then for this form $f$ we must have $p^2\mid \Delta$
  and $p\nmid\partial\Delta/\partial e$ (for otherwise $f$ would
  belong to~$\W_p^{(1)}$). Since $R_X^{(\epsilon)}$ is a homogeneously
  expanding region in $V_\R=\R^5$ with each side growing at the order
  of $X^{1/6}$, there are $O(X^{4/6})$ possibilities for a quadruple
  $(a_0,b_0,c_0,d_0)$ such that $f(x,y)\in R_X^{(\epsilon)}\cap V_\Z$
  for some $e_0$. Given fixed values of $a_0$, $b_0$, $c_0$, and
  $d_0$, there are at most $3$ choices for the residue of $e_0$
  (mod~$p$) such that $p\mid\Delta$. Since
  $p\nmid \partial\Delta/\partial e$, each such residue modulo $p$ has
  a unique lift modulo $p^2$ such that $p^2\mid\Delta$.  Hence, we
  have
\begin{equation}\label{equnifbqsecond}
\#\{R_X^{(\epsilon)}\cap\W_p^{(2)}(V))\}=O(\max\{X^{5/6}/p^2,X^{4/6}\}),
\end{equation}
where we may use the first estimate for $p\leq X^{1/12}$ and the
second estimate for $p>X^{1/12}$. Since there are $O(X^{1/6}/\log X)$
primes in the range $[1,X^{1/6}]$, and since
$\sum_{p>X^{1/6}}1/p^2=O(1/(X^{1/6}\log X))$, we obtain
$$
\#\left\{R_X^{(\epsilon)}\textstyle{\bigcap}(\displaystyle\cup_{p>M}\W^{(2)}_p(V))\right\}=O(\sum_{p>M}\#\{R_X^{(\epsilon)}\cap\W_p^{(2)}(V)\})=O(X^{5/6}/\log M)
$$
by using (\ref{equnifbqsecond}) to estimate
$\#\{R_X^{(\epsilon)}\cap\W_p^{(2)}(V)\}$ when $p< X^{1/6}$,
and using 
(\ref{equnifbqfirst}) 
when 
$p\geq X^{1/6}$.
\end{proof}

Using the above two uniformity estimates, we obtain a proof of Theorem~\ref{thunifbqelem}:

\vspace{0.1in}
\noindent{\bf Proof of Theorem \ref{thunifbqelem}:} Let $R(X)$ denote $\cup_i
R^{(i)}(X)$. By the results of \S2.1, we have:
  \begin{equation}\label{eqfinalunif}
    \begin{array}{rcl}
      N(\cup_{p>M}\W_p(V),X)&\leq& \#\{\FF_{\PGL_2}\cdot R(X)\bigcap(\displaystyle\cup_{p>M}\W_p(V))\cap V_\Z^\irr\}\\[0.1in]
      &\leq&\#\{\FF^{(\epsilon)}_{\PGL_2}\cdot R(X)\bigcap(\displaystyle\cup_{p>M}\W_p(V))\}+\#\{(\FF_{\PGL_2}\backslash\FF^{(\epsilon)}_{\PGL_2})\cdot R(X)\cap V_\Z^\irr)\}.      
    \end{array}
  \end{equation}
  By Theorems \ref{thunifbqelem1} and \ref{thunifbqelem2}, the first
  term in the second line of (\ref{eqfinalunif}) is bounded by
  $O(X^{5/6}/\log\,M+X^{2/3})$. The results of \S2.3 and \S2.4 imply
  that the second term is bounded by
  $\Vol((\FF_{\PGL_2}-\FF^{(\epsilon)}_{\PGL_2})\cdot
  R(X))=O(\epsilon X^{5/6})$.  Since this holds for all
  $\epsilon>0$, the theorem follows. $\Box$

\subsection{A squarefree sieve}\label{squarefree}

For the applications, we require a more general congruence version of our
counting theorem for binary quartic forms, namely, one which allows
appropriate infinite sets of congruence conditions to be imposed and
which also allows weighted counts of lattice points (where weights are
also assigned by congruence conditions). More precisely, we say that a
function $\phi:V_\Z\to[0,1]\subset\R$ is {\it defined by congruence
  conditions} if, for all primes $p$, there exist functions
$\phi_p:V_{\Z_p}\to[0,1]$ satisfying the following conditions:
\begin{itemize}
\item[(1)] For all $f\in V_\Z$, the product $\prod_p\phi_p(f)$
  converges to $\phi(f)$.
\item[(2)] For each prime $p$, the function $\phi_p$ is 
locally constant outside some closed set $S_p \subset V_{\Z_p}$ of measure zero.
\end{itemize}
Such a function $\phi$ is called {\it acceptable} if, for sufficiently
large primes $p$, we have $\phi_{p}(f)=1$ whenever $p^2\nmid
\Delta(f)$.  For example, the
characteristic function of the set of integral binary quartic forms
having squarefree discriminant is an acceptable function.

We then have the following version of Theorem~\ref{cong3}, in which we
allow weights to be defined by certain infinite sets of congruence
conditions:
\begin{theorem}\label{thsquarefreebq}
  Let $\phi:V_\Z\to[0,1]$ be an acceptable function that is defined by
  congruence conditions via the local functions $\phi_{p}:V_{\Z_p}\to[0,1]$. Then, with
  notation as in Theorem~$\ref{cong3}$, we have:
\begin{equation}
N_\phi(V_\Z^{(i)};X)
  = N(V_\Z^{(i)};X)
  \prod_{p} \int_{f\in V_{\Z_{p}}}\phi_{p}(f)\,df+o(X^{5/6}).
\end{equation}
\end{theorem}
\begin{proof}
  Since $\phi_p$ is locally constant outside some set of measure zero,
  there exists an increasing sequence of functions
  $\psi_{p,1}\leq\psi_{p,2}\leq\cdots$ that are bounded
  above by and converge pointwise to 
  $\phi_p$, and a decreasing sequence of functions
  $1=\psi'_{p,0}\geq\psi'_{p,1}\geq\psi'_{p,2}\geq\cdots$ that are bounded
  below by and converge pointwise to~$\phi_p$, 
such that $\psi_{p,n}$ and $\psi'_{p,n}$
  are defined on $V_{\Z_p}$ by congruence conditions modulo $p^n$.
It will also be convenient in the formulas that follow to define $\psi_{p,0}$ to equal the constant function $1$ on $V_{\Z_p}$.

 By  the dominated convergence theorem, we have
  \begin{equation}\label{eqsqflimint}
\lim_{n\to\infty}\int_{V_{\Z_p}}\psi_{p,n}(f)df=\lim_{n\to\infty}\int_{V_{\Z_p}}\psi'_{p,n}(f)df=\int_{V_{\Z_p}}\phi_{p}(f)df.
  \end{equation}
Furthermore, since $\phi$ is acceptable we have
\begin{equation}\label{eqsqfden}
1-\int_{V_{\Z_p}}\phi_p(f)df\leq \int_{\substack{f\in V_{\Z_p}\\p^2\mid\Delta(f)}}df\ll p^{-2}
\end{equation}
for sufficiently large $p$ (see, for example, \cite[Proof of Theorem 3.2]{BPS}).

For a fixed integer
  $Y$, let $N_\psi^Y(V_\Z^{(i)};X)$ (resp.\ $N_{\psi'}^Y(V_\Z^{(i)};X)$) denote the number of irreducible
  $\GL_2(\Z)$-orbits in $V_\Z^{(i)}$ having height bounded by $X$,
  where each orbit $\GL_2(\Z)\cdot f$ is counted with weight
  $$\displaystyle\prod_{p}\psi_{p,\lfloor Y/p\rfloor}(f) \;\;\;\;\Bigl(\mbox{resp}.\ \displaystyle\prod_{p}\psi'_{p,\lfloor Y/p\rfloor}(f)\Bigr).$$ 
  The function $\lfloor Y/p\rfloor$ is chosen to take nonzero values
  only for finitely many primes $p$ for any fixed $Y$.  Therefore, it
  follows from Theorem~\ref{cong3} that, for any fixed $Y$, we have
  \[\limsup_{X\rightarrow\infty}\frac{N_{\phi}(V^{(i)}_\Z;X)}{X^{5/6}}\leq\limsup_{X\rightarrow\infty}\frac{N^Y_{\psi'}(V_\Z^{(i)};X)}{X^{5/6}}=\lim_{X\rightarrow\infty}\frac{N(V_\Z^{(i)};X)}{X^{5/6}}\prod_{p}\int_{f\in V_{\Z_{p}}}\psi'_{p,\lfloor Y/p\rfloor}(f)\,df.\]
Equation (\ref{eqsqfden}) implies that the product $\prod_p\int_{V_{\Z_p}}\phi_p(f)df$ converges.
Letting $Y$ tend to infinity, we have by (\ref{eqsqflimint}) that
  \begin{equation}\label{eqsieveproof1}
\limsup_{X\rightarrow\infty}\frac{N_\phi(V_\Z^{(i)};X)}{X^{5/6}}\leq
    \lim_{X\rightarrow\infty}\frac{N(V_\Z^{(i)};X)}{X^{5/6}}\prod_{p}\int_{f\in V_{\Z_{p}}}\phi_{p}(f)\,df.
\end{equation}

We now obtain a lower bound using Theorem~\ref{thunifbqelem}.  For
sufficiently large $p$ and $n\geq 1$, we have
$\psi_{p,n}(f)=\phi_p(f)=1$ unless $p^2\mid \Delta(f)$. Thus, for sufficiently large $Y$, we have
\begin{eqnarray*}
  \liminf_{X\rightarrow\infty}\frac{N_{\phi}(V^{(i)}_\Z;X)}{X^{5/6}}&\geq&\liminf_{X\rightarrow\infty}\Bigl[\frac{N^Y_\psi(V_\Z^{(i)};X)}{X^{5/6}}-\frac{O({N(\displaystyle\cup_{p>Y}\W_p(V);X)})}{X^{5/6}}\Bigr]\\[.1in] &=&
  \lim_{X\rightarrow\infty}\frac{N(V_\Z^{(i)};X)}{X^{5/6}}\cdot \prod_{p}\int_{f\in
    V_{\Z_{p}}}\psi_{p,\lfloor
    Y/p\rfloor}(f)\,df
-O(1/\log\, Y),
\end{eqnarray*}
where the first inequality follows because $\phi$ is
an upper bound for $\psi_{p,n}$ unless $n=0$, and the last equality follows
from Theorems~\ref{cong3} and \ref{thunifbqelem}. Taking the limit as $Y$ tends to infinity then yields
  \begin{equation}\label{eqsieveproof2}
\displaystyle\liminf_{X\rightarrow\infty}\frac{N_\phi(V_\Z^{(i)};X)}{X^{5/6}}
= 
\displaystyle
   \lim_{X\rightarrow\infty}
 \frac{N(V_\Z^{(i)};X)}{X^{5/6}}\cdot\prod_{p}\int_{f\in V_{\Z_{p}}}\phi_{p}(f)\,df
\end{equation}
where we use (\ref{eqsqfden}) to exchange the limit (in $Y$) and product, and
(\ref{eqsqflimint}) to exchange the limit (in $Y$) and integral. The
theorem now follows from (\ref{eqsieveproof1}) and
(\ref{eqsieveproof2}).
\end{proof}

\subsection{Proofs of auxiliary results (Lemma~\ref{rstab}, 
Lemma~\ref{gl2zbigstab}, and Theorem~\ref{eligible})}\label{secaux}

The proofs of the auxiliary results referred to in the title all turn
out to have natural interpretations in terms of the monic cubic
resolvent forms of binary quartic forms, as discussed in
\S\ref{secunif}.  More precisely, a {\it monic binary cubic form}
$g(x,y)$ is defined as a binary cubic form $g(x,y)$ whose leading
coefficient as a polynomial in $x$ is equal to $1$, i.e., it is of the
form $x^3+rx^2y+sxy^2+ty^3$.  We denote the space of binary cubic
forms over $\Z$ by $U_\Z$, and the subset of monic binary cubic forms
over $\Z$ by $U_{\Z,1}\subset U_\Z$. Note that if $(A,B)\in W_\Z$,
then $4\,\Det(Ax-By)\in U_\Z$, and that if $(A_1,B)\in W_{\Z,1}$, then
$4\,\Det(A_1x-By)\in U_{\Z,1}$.

The group $F_{\Z,1}$ acts naturally on
$U_{\Z,1}$ via $\gamma \cdot g(x,y)=g((x,y)\cdot\gamma)$.
If $g(x,y)=x^3+rx^2y+sxy^2+ty^3$, then one easily sees that the quantities 
\begin{equation}\label{eqmcij}\begin{array}{rcl}I(g)&:=&r^2-3s,\\[.025in]
J(g)&:=&-2r^3+9rs-27t\end{array}\end{equation}
are invariant under the action of $F_{\Z,1}$.  The discriminant
$\Delta(g)$ of the binary cubic form $g$ can be expressed in terms of these
basic invariants $I(g)$ and $J(g)$, namely,
$\Delta(g)=(4I(g)^3-J(g)^2)/27$.
We again define the {\it height} of $g$ by
$$H(g):=H(I,J)=\max\{|I(g)^3|,J(g)^2/4\}.$$
If $F_{\Q,1}$ denotes the group of lower triangular matrices in
$\SL_2(\Q)$ with 1's on the diagonal, then by using an
$F_{\Q,1}$-transformation to clear out the $x^2y$-coefficient, we see that $g(x,y)$ is
$F_{\Q,1}$-equivalent to the monic binary cubic form
$h(x,y)=x^3-\frac{I(f)}{3}xy^2-\frac{J(f)}{27}y^3$.  

If $f\in V_\Z$ is an integral binary quartic form, then as in
\S\ref{secunif} we define the monic {\it cubic resolvent form} of $f$
by $g(x,y)=4\,\Det(A_1x-B_fy)$, where $(A_1,B_f)$ is the image of $f$
under the map $\phi$ defined in~(\ref{vtow}). It is easy to check that
$I(f)=I(g)$ and $J(f)=J(g)$. 
The elliptic curve
$E_f:z^2=g(x,1)$ (which we may also write as 
$z^2=x^3-\frac{I(f)}{3}x-\frac{J(f)}{27}$) turns out to be
the Jacobian of the genus one curve $C_f$ in weighted projective space
$\P(1,1,2)$ determined by the equation $z^2=f(x,y)$; furthermore, the
stabilizer of~$f$ in $\PGL_2(\Q)$ is isomorphic to $E_f(\Q)[2]$ (see
Theorem~\ref{eqtwistedaction}).  This connection between $f$ and $E_f$
will be of key importance in the next section.

We first use this connection to prove Lemma~\ref{gl2zbigstab},
which states that the stabilizer
in $\GL_2(\R)$ of $f\in V_\R$ is $8$ or $4$ in accordance with whether the
discriminant of $f$ is positive or negative, respectively.

\vspace{.125in}

\noindent
{\bf Proof of Lemma \ref{rstab}:} Consider the action of
$\PGL_2(\R)$ on $V_\R$ defined by (\ref{eqtwistedaction}). If $f\in
V_\R$ has nonzero discriminant, then Theorem \ref{thparame2e} in
Section~3 (which does not rely on the results of this section) asserts that
$\Stab_{\PGL_2(\R)}(f)$ is isomorphic to $E(\R)[2]$, where $E$ is the
elliptic curve given by
$y^2=x^3-\frac{I(f)}{3}x-\frac{J(f)}{27}$. Therefore,
$\#\Stab_{\PGL_2(\R)}(f)$ is equal
to $2$ if $\Delta(f)<0$ and
equal to $4$ if $\Delta(f)>0$.

Now if $\gamma\in\GL_2(\R)$
stabilizes $f\in V_\R$ under the usual action (defined in (\ref{equntwistedaction})), 
then since $I(\gamma\cdot f)=(\det
\gamma)^4I(f)$ and $J(\gamma\cdot f)=(\det \gamma)^6J(f)$, we see that
$\det \gamma=\pm1$. Hence the
image of $\gamma$ in
$\PGL_2(\R)$ also stabilizes $f$. Since there are two elements in the
center of $\GL_2(\R)$ that stabilize $f$, the size of
the stabilizer in $\GL_2(\R)$ of an element $f\in V^{(i)}_\R$ is $4$ when $i=1$ (equivalently $\Delta(f)<0$) and
$8$ when $i=0$ or $2$ (equivalently $\Delta(f)>0$), as desired.  $\Box$\vspace{.125 in}

To prove Lemma~\ref{gl2zbigstab}, which states that the number of
$\GL_2(\Z)$-orbits on binary quartic forms having bounded height and a
nontrivial stabilizer in $\PGL_2(\Q)$ is negligible, we use the
following lemma:

\begin{lemma}\label{mcredest}
  The number of $F_{\Z,1}$-orbits on monic integral binary cubic forms
  $g$ such that $g$ is reducible over $\Q$ and $H(g)<X$ is $O(X^{1/2+\epsilon})$.
\end{lemma}

\begin{proof}
  First, we note that if $g(x,y)=x^3+rx^2y+sxy^2+ty^3\in U_{\Z,1}$, then by
  replacing $g$ with an $F_{\Z,1}$-translate if necessary we may assume
  that $r\in\{-1,0,1\}$. Throughout the rest of this proof, we will
  assume that this is the case. If $g$ is such that $H(g)<X$, then
  since $|I(g)|^3=|r^2-3s|^3\leq H(g)<X$, we see that $s=O(X^{1/3})$.
  Since $J(g)^2/4=(2r^3+9rs-27t)^2/4\leq H(g)<X$, this in turn implies
  that $t=O(X^{1/2})$.

  Let us now count such forms $g$ that are reducible.
  If $g(x,y)=x^3+rx^2y+sxy^2+ty^3$ satisfies $t=0$ (and $r\in\{-1,0,1\}$), then $g$ is
  reducible, and the number of such forms $g$
  with $H(g)<X$ is the number of possible values for $r$ and
  $s$, namely $3\cdot O(X^{1/3})=O(X^{1/3})$.

  Next, we consider those reducible forms
  $g(x,y)=x^3+rx^2y+sxy^2+ty^3$ satisfying $H(g)<X$, $r\in\{-1,0,1\}$
  and $t \neq 0$. If $x-my$ is a factor of $g$, then $m\mid t$.
  Therefore, if we fix $t \neq 0$, then there are at most
  $t^{\epsilon} = O(X^{\epsilon})$ choices for $m$. Moreover, once
  $r$, $t,$ and $m$ are fixed, then setting $g(m,1)$ equal to 0
  determines $s$.  Since $t=O(X^{1/2})$, and there are at most 3
  possible values for $r$, it follows that there are at most
  $O(X^{1/2+\epsilon})$ such reducible forms $g$ with height
  less than $X$.
\end{proof}

\noindent{\bf Proof of Lemma \ref{gl2zbigstab}:} Suppose an integral binary
quartic form $f$ has a stabilizer of size at least $2$
in $\PGL_2(\Q)$. Then Theorem \ref{thparame2e} asserts that $E(\Q)[2]$ is
nontrivial, where $E$ is given by
$y^2=x^3-\frac{I(f)}3x-\frac{J(f)}{27}$. This implies that the cubic
resolvent form $g$ of $f$ is reducible over $\Q$. If we further assume
that $H(f)<X$, then Lemma \ref{mcredest} implies that there are at
most $O(X^{1/2+\epsilon})$ choices for the $F_{\Z,1}$-orbit of $g$.

Now, if the $\GL_2(\Z)$-orbit of a reducible integral binary cubic
form $g$ having height $X$ is fixed, then \cite[Proof of
Lemma~12]{dodqf} implies that the number of
$\GL_2(\Z)\times\SL_3(\Z)$-orbits on $W_\Z$ having $g$ as a cubic
resolvent form is bounded by $O(X^{1/4})$. In conjunction with Proposition
\ref{atmosttwelve}, this implies that the number of
$\PGL_2(\Z)$-orbits on $V_\Z$ having $g$ has a cubic resolvent form is
also at most $O(X^{1/4})$. Therefore, the number of
$\PGL_2(\Z)$-orbits on $V_\Z$ having a nontrivial stabilizer in
$\PGL_2(\Q)$ and height less than $X$ is bounded by
$O(X^{1/4}X^{1/2+\epsilon})=O(X^{3/4+\epsilon})$. This concludes the
proof of Lemma \ref{gl2zbigstab}. $\Box$\vspace{.1 in}

Finally, we determine when a pair of invariants $(I,J)\in\Z\times\Z$
is eligible,
thus proving Theorem \ref{eligible}.

\vspace{.125in}

\noindent
{\bf Proof of Theorem \ref{eligible}:} If an integral binary quartic
form has invariants equal to $I$ and $J$, then its cubic resolvent
form also has invariants equal to $I$ and $J$. 
Conversely, suppose an integral pair $(I,J)$ occurs as the invariants
of an integral monic binary cubic form
$g(x,y)=x^3+rx^2y+sxy^2+ty^3$. Then one checks that the cubic
resolvent form of the binary quartic form
$f(x,y)=x^3y+rx^2y^2+sxy^3+ty^4$ is equal to $g$, and so $f$ has
invariants equal to $I$ and $J$.  Therefore the pair $(I,J)$ is
eligible.  Hence, to prove Theorem \ref{eligible}, it suffices to answer the simpler question: which integral pairs $(I,J)$
occur as invariants of integral monic binary cubic forms?

  Suppose the integral monic binary cubic form
  $g(x,y)=x^3+rx^2y+sxy^2+ty^3\in U_{\Z,1}$ has invariants $I$ and $J$. 
By replacing $g$ with an $F_{\Z,1}$-translate if necessary, we may assume that $r\in\{-1,0,1\}$. This does not change the invariants $I$ and $J$.
  If $I \equiv 0 \pmod3$ then $r=0$, implying that
  $27\mid J$.  
  This is condition~(a) in Theorem \ref{eligible}.

  If $I$ is not divisible by $3$, then $r$ equals $1$ or $-1$ and we
  have $I \equiv 1 \pmod3$. Thus $I$ must be congruent to $1$, $4,$ or $7
  \pmod9$, which happens exactly when $s$ is congruent to $0$, $2,$ or $1
  \pmod 3$, respectively.  Because $r^2=1$, we see that $J \equiv
  r(9s-2) \pmod{27}$. It follows that $I\equiv 1$, $4$, $7\pmod 9$
  corresponds to $J\equiv \pm2$, $\pm16$, $\pm7\pmod{27}$, respectively,
  yielding conditions (b), (c), and (d).

  Therefore, if a pair $(I,J)$ occurs as the invariants of an integral
  monic binary cubic form, then it must satisfy one of the conditions
  of Theorem \ref{eligible}.  The converse also follows easily by
  reversing the above arguments. This concludes the proof of Theorem \ref{eligible}. $\Box$

\section{The average size of the $2$-Selmer groups of elliptic curves}

Recall that every elliptic curve $E$ over $\Q$ can be written
in the form
\begin{equation}\label{eqellip}
E_{A,B}: y^2=x^3+Ax+B,
\end{equation}
where $A,B\in\Z$ and $p^4\nmid A$ if $p^6\mid B$.  For any elliptic
curve $E=E_{A,B}$ over $\Q$ written in the form (\ref{eqellip}), we
define the quantities $I=I(E)$ and $J=J(E)$ by
\begin{equation}\label{eqIJE}
\begin{array}{rcl}
I(E)&:=&-3A,\\[.02in]
J(E)&:=&-27B,\end{array}\end{equation}
and denote the curve $E_{A,B}$ also by $E^{I,J}$.
The height of $E_{A,B}=E^{I,J}$ is then defined by
\begin{equation*}
H(E_{A,B})=\max\{4|A^3|,27B^2\}=\frac{4}{27}\max\{I(E)^3,J(E)^2/4\}.
\end{equation*}
In this section, we shall work with the slightly different height
$H'(E)$ defined by
\begin{equation}
H'(E):=H(I(E),J(E))=\max\{|I(E)|^3,J(E)^2/4\}, 
\end{equation}
so that the height agrees with the height defined for binary quartic forms
in (\ref{hbqdef}).  Note that $H$ and $H'$ only differ by a constant factor;
namely, for every elliptic curve $E$ over $\Q$ we have
$27H(E)=4H'(E)$.

Our purpose in this section is to prove Theorem \ref{ellipcong} by
computing the average size of the $2$-Selmer group of
elliptic curves $E/\Q$ when these curves are ordered by their heights (note
that the two heights $H$ and $H'$ give the same ordering on every set
of elliptic curves).  Theorem \ref{mainellip}, being a special case of
Theorem~\ref{ellipcong}, will then follow.

In fact, we prove a statement stronger than Theorem \ref{ellipcong}.
To state this result, we need some notation.  For each prime $p$, let
$\Sigma_p$ be a closed subset of $\Z_p^2\backslash\{\Delta\neq 0\}$
whose boundary has measure $0$. To such a collection $(\Sigma_p)_p$,
we associate the set $F_\Sigma$ of elliptic curves over $\Q$, where
$E^{I,J}\in F_\Sigma$ if and only if $(I,J)\in\Sigma_p$ for all
$p$. We then say that $F_\Sigma$ is a family of elliptic curves over
$\Q$ that is {\it defined by congruence conditions}. We can also
impose ``congruence conditions at infinity'' on $F_\Sigma$ by
insisting that an elliptic curve $E^{I,J}$ belongs to $F_\Sigma$ if
and only if $(I,J)$ belongs to $\Sigma_\infty$, where $\Sigma_\infty$
is equal to $\{(I,J)\in\R^2:\Delta(I,J)>0\}$,
$\{(I,J)\in\R^2:\Delta(I,J)<0\}$, or $\{(I,J)\in\R^2:\Delta(I,J)\neq
0\}$.

If $F$ is any nonempty family of elliptic curves over $\Q$ defined by
congruence conditions, then let
$\Inv(F)$ denote the set $\{(I(E),J(E)):E\in F\}$. We define
$\Inv_p(F)$ to be the set of those elements $(I,J)$ in the $p$-adic
closure of $\Inv(F)\subset\Z_p^2$ such that
$\Delta(I,J):=(4I^3-J^2)/27\neq 0$. Also, we define $\Inv_\infty(F)$ by
$\{(I,J)\in\R^2:\Delta(I,J)>0\}$,
$\{(I,J)\in\R^2:\Delta(I,J)<0\}$, or $\{(I,J)\in\R^2:\Delta(I,J)\neq
0\}$ in accordance with whether $F$ contains only curves of positive
discriminant, negative discriminant, or both.  A family $F$ of
elliptic curves defined by congruence conditions is then said
to be {\it large} if, for all but finitely many primes~$p$, the
set $\Inv_p(F)$ contains all pairs $(I,J)\in\Z_p\times\Z_p$ such that
$p^2\nmid\Delta(I,J)$.  
In this section, we prove the following
strengthening of Theorem \ref{ellipcong}.

\begin{theorem}\label{ellipall}
  When all elliptic curves $E$ in any large family are ordered by height,
  the average size of the $2$-Selmer group $S_2(E)$ is~$3$.
\end{theorem}

Note that the family of all elliptic curves is large. So too is
the family of elliptic curves $E:y^2=g(x)$ defined by finitely many
congruence conditions on the coefficients of $g$. Thus
Theorems~\ref{mainellip} and \ref{ellipcong} indeed follow
from~Theorem \ref{ellipall}.  Finally, we note that the family of all
semistable elliptic curves is also large.

\subsection{Preliminaries on binary quartic forms and 2-coverings of
  elliptic curves}\label{secminmod}

The key to proving Theorem~\ref{ellipall} is the use of a classical
correspondence between elements in the 2-Selmer group of an elliptic
curve $E^{I,J}$ over $\Q$ and locally soluble integral binary quartic
forms having invariants $2^4I$ and $2^6J$. This correspondence was
originally introduced by Birch
and Swinnerton-Dyer~\cite{BSD}, and was developed further by
Cremona~\cite{Cbook} (see also~\cite{CS}, \cite{CF}, and \cite{BhHo}).
We collect here the results that we will need on this correspondence.
Throughout this section, we use the action of $\PGL_2$ on $V$ as
defined by (\ref{eqtwistedaction}).

We say that a binary quartic form over a field $K$ is {\it $K$-soluble}
if the equation $z^2=f(x,y)$ has a solution with $x,y,z\in K$ and
$(x,y)\neq (0,0)$.
The first paragraph of the following
theorem is contained in~\cite[Proposition 2.2]{CS}, while the second
follows from~\cite[\S3--5 and Remark 1]{CF}.
(For more details, see \cite[\S4.1]{BhHo}.)
\begin{theorem}\label{thparame2e}
  Let $K$ be a field having characteristic not $2$ or $3$. Let
  $E:y^2=x^3-\frac{I}3x-\frac{J}{27}$ be an elliptic curve over $K$.
  Then there exists a bijection between elements in $E(K)/2E(K)$ and
  $\PGL_2(K)$-orbits of $K\!$-soluble binary quartic forms having
  invariants $I$ and $J$, given by
$$
(\xi,\eta)+2E(K)\mapsto \PGL_2(K)\cdot\left(\frac{1}{4}x^4-\frac{3}{2}\xi x^2y^2+2\eta xy^3+\left(\frac{I}{3}-\frac{3}{4}\xi^2\right)y^4\right).
$$
Under this bijection, the identity element in
$E(K)/2E(K)$ corresponds to the $\PGL_2(K)$-orbit of binary quartic
forms having a linear factor over $K$.

Furthermore, the stabilizer in $\PGL_2(K)$ of any $($not necessarily
$K\!$-soluble$)$ binary quartic form $f$ in $V_K$, having nonzero
discriminant and invariants $I$ and $J$, is isomorphic to $E(K)[2]$,
where $E$ is the elliptic curve defined by
$y^2=x^3-\frac{I}{3}x-\frac{J}{27}$.
\end{theorem}

Next, recall that a 
binary quartic form $f\in V_\Q$ is called {\it locally soluble}
if $f$ is $\R$-soluble and $\Q_p$-soluble for all primes $p$. We then have
the following proposition (see \cite[Lemma 2]{BSD} and the discussion
following it).
\begin{proposition}\label{propparamrational}
  Let $E:y^2=x^3-\frac{I}{3}x-\frac{J}{27}$ be an elliptic curve over
  $\Q$. Then there exists a bijection between isomorphism classes of
  locally soluble $2$-coverings of $E$ and $\PGL_2(\Q)$-orbits of
  locally soluble binary quartic forms in $V_\Q$ having invariants
  $I$ and $J$.

  Furthermore, the set of rational binary quartic forms having a
  rational linear factor and invariants equal to $I$ and $J$ lie
  in a single $\PGL_2(\Q)$-orbit, and this orbit corresponds to
  the identity element in the $2$-Selmer group of $E$.
\end{proposition}

In order to prove Theorem \ref{ellipall}, we will also require the
following lemma, which follows from Lemmas~3, 4, and 5 of \cite{BSD}.

\begin{lemma}\label{lemreduction}
  Let $f\in V_\Q$ be a locally soluble binary quartic form having
  integral invariants $I$ and $J$ such that $(2^4\cdot 3) \mid I$ and
  $(2^6\cdot 3^3) \mid J$. Then $f$ is $\PGL_2(\Q)$-equivalent to an integral
  binary quartic form.
\end{lemma}

Since $E=E^{I,J}$ is also isomorphic to the elliptic curve defined by
$y^2=x^3-\frac{2^4I(E)}{3}x-\frac{2^6J(E)}{27}$,
Proposition~\ref{propparamrational} and Lemma~\ref{lemreduction} now imply
the following theorem:

\begin{theorem}\label{2spar}
  Let $E=E^{I,J}$ be an elliptic curve over $\Q$. Then the elements
  of the $2$-Selmer group of $E$ are in one-to-one correspondence with
  $\PGL_2(\Q)$-equivalence classes of locally soluble integral binary
  quartic forms having invariants equal to $2^4I$ and $2^6J$.

  Furthermore, the set of integral binary quartic forms that have a
  rational linear factor and invariants equal to $2^4I$ and $2^6J$ lie
  in one $\PGL_2(\Q)$-equivalence class, and this class corresponds to
  the identity element in the $2$-Selmer group of $E$.
\end{theorem}

\subsection{A weighted set $S(F)$ of integral binary quartic forms
  associated to a large family $F$ of elliptic curves}
\label{sec5main}

Theorem \ref{2spar} asserts that nonidentity elements in the
$2$-Selmer group of an elliptic curve $E^{I,J}$ over $\Q$ are in
bijective correspondence with $\PGL_2(\Q)$-equivalence classes of
locally soluble integral binary quartic forms having invariants $2^4I$
and $2^6J$ that do not possess a rational linear factor. In~\S2, we
computed the asymptotic number of $\GL_2(\Z)$-orbits of irreducible
integral binary quartic forms having bounded height. By
Lemma~\ref{reducible}, the number of $\GL_2(\Z)$-orbits of binary
quartic forms of bounded height that are the product of two
irreducible integral binary quadratic forms is
negligible. Furthermore, $\GL_2(\Z)$-orbits on $V_\Z$ are exactly the
same as $\PGL_2(\Z)$-orbits on $V_\Z$. Therefore, the same asymptotic
formula in Theorem~\ref{2spar} holds also for the number of $\PGL_2(\Z)$-orbits of integral
binary quartic forms having bounded height and no rational linear
factor.

In order to adapt the latter results to compute the number of
$\PGL_2(\Q)$-equivalence classes of locally soluble integral binary
quartic forms having bounded height and no rational linear factor, we
need to count each $\PGL_2(\Z)$ orbit, $\PGL_2(\Z)\cdot f$,
weighted by $1/n(f)$, where $n(f)$ is equal to the number of
$\PGL_2(\Z)$-orbits inside the $\PGL_2(\Q)$-equivalence class of $f$
in $V_\Z$. For this purpose, it suffices to count the number of
$\PGL_2(\Z)$-orbits of locally soluble integral binary quartic forms
having bounded height and no rational linear factor where each orbit
$\PGL_2(\Z)\cdot f$ is weighted by $1/m(f)$, where
$$m(f):=\displaystyle\sum_{f'\in B(f)}\frac{\#\Aut_\Q(f')}{\#\Aut_\Z(f')}=\displaystyle\sum_{f'\in B(f)}\frac{\#\Aut_\Q(f)}{\#\Aut_\Z(f')};$$ 
here $B(f)$ denotes a set of representatives for the action of $\PGL_2(\Z)$
on the $\PGL_2(\Q)$-equivalence class of $f$ in $V_\Z$, and
$\Aut_\Q(f)$ (resp.\ $\Aut_\Z(f)$) denotes the stabilizer of $f$ in
$\PGL_2(\Q)$ (resp.\ $\PGL_2(\Z)$). The reason it suffices to weight
by $1/m(f)$ instead of $1/n(f)$ is that, by Lemma \ref{gl2zbigstab},
all but a negligible number of $\PGL_2(\Z)$-orbits of integral binary
quartic forms with nonzero discriminant and bounded height have
trivial stabilizer in $\PGL_2(\Q)$; thus all but a negligible number
of $\PGL_2(\Z)$-equivalence classes of integral binary quartic forms
with nonzero discriminant and bounded height satisfy $m(f)=n(f)$.  

Let us use $S(F)$ to denote the set of all locally soluble integral binary
quartic forms having invariants $2^4I$ and $2^6J$, where
$(I,J)\in \Inv(F)$.  Assign to each element $f\in S(F)$ the weight
$1/m(f)$.  Then we conclude that the weighted number of irreducible
$\PGL_2(\Z)$-orbits of height less than $X$ in $S(F)$ is asymptotically
equal to the number of nonidentity 2-Selmer elements of all elliptic
curves of height less than $X$ in $F$.  In the remainder of this
section, our goal is therefore to count the weighted number of
irreducible orbits in $S(F)$ having bounded height. 

The global weights $m(f)$ (as opposed to $n(f)$) are useful for the
following reason.  For a prime $p$ and a binary quartic form $f\in
V_{\Z_p}$, define $m_p(f)$ by
$$m_p(f):=\displaystyle\sum_{f'\in B_p(f)}\frac{\#\Aut_{\Q_p}(f')}{\#\Aut_{\Z_p}(f')}=\displaystyle\sum_{f'\in B_p(f)}\frac{\#\Aut_{\Q_p}(f)}{\#\Aut_{\Z_p}(f')},$$
where $B_p(f)$ denotes a set of representatives for the action of
$\PGL_2(\Z_p)$ on the $\PGL_2(\Q_p)$-equivalence class of $f$ in
$V_{\Z_p}$, and $\Aut_{\Q_p}(f)$ (resp.\ $\Aut_{\Z_p}(f)$) denotes the
stabilizer of $f$ in $\PGL_2(\Q_p)$ (resp.\ $\PGL_2(\Z_p)$).

Then we have the following proposition:

\begin{proposition}\label{propmassform}
  Suppose $f\in V_\Z$ has nonzero discriminant. Then $m(f)=\prod_pm_p(f)$.
\end{proposition}

\begin{proof}
  Let $\PGL_2(\Q)_f$ (resp.\ $\PGL_2(\Q_p)_f$) denote the set of
  elements $\gamma\in\PGL_2(\Q)$ (resp.\ $\PGL_2(\Q_p)$) such that
  $\gamma\cdot f\in V_\Z$ (resp.\ $V_{\Z_p}$).
  Then we have a natural map from $\PGL_2(\Q)_f$ to the set of
  $\PGL_2(\Z)$-orbits on the $\PGL_2(\Q)$-equivalence class of $f$ in
  $V_\Z$ via $\gamma\mapsto\PGL_2(\Z)\gamma\cdot f$. Two elements in
  $\PGL_2(\Q)_f$ map to the same orbit if and only if they map to the same element in the double coset space
$$
\PGL_2(\Z)\backslash\PGL_2(\Q)_f/\Aut_\Q(f).
$$
Thus, the number of elements in $\PGL_2(\Z)\backslash\PGL_2(\Q)_f$
that map to a fixed orbit $\PGL_2(\Z)\cdot f'$ is equal to
$\#\Aut_{\Q}(f)/\#\Aut_{\Z}(f')$, implying that
  $$\#[\PGL_2(\Z)\backslash\PGL_2(\Q)_f]=\displaystyle\sum_{f'\in B(f)}\frac{\#\Aut_{\Q}(f)}{\#\Aut_{\Z}(f')}=m(f).$$
  Similarly, we have that
  $$\#[\PGL_2(\Z_p)\backslash\PGL_2(\Q_p)_f]=\displaystyle\sum_{f'\in B_p(f)}\frac{\#\Aut_{\Q_p}(f)}{\#\Aut_{\Z_p}(f')}=m_p(f).$$

  Now we consider the map
  $$\tau:\PGL_2(\Z)\backslash\PGL_2(\Q)_f\to\prod_p\PGL_2(\Z_p)\backslash\PGL_2(\Q_p)_f$$
  given by the diagonal embedding. Since $\PGL_2(\Q_p)_f=\PGL_2(\Z_p)$
  for all primes $p$ not dividing the discriminant of $f$ (see the proof of
  Proposition~\ref{unifellip} for a stronger result), the product
  $\prod_p\PGL_2(\Z_p)\backslash\PGL_2(\Q_p)_f$ is in fact a finite product.
  It is easy to see that $\tau$ is well-defined and injective. (For
  injectivity, note that if $\gamma_1$ and $\gamma_2$ are elements in
  $\PGL_2(\Q)_f$ that map to the same element under $\tau$, then
  $\gamma_1\gamma_2^{-1}$ is an element of $\PGL_2(\Q)$ and of
  $\PGL_2(\Z_p)$ for all $p$. This implies that
  $\gamma_1\gamma_2^{-1}\in\PGL_2(\Z)$, as desired.)

  The group $\PGL_2(\Q)$ has class number 1 (see \cite[Chapter
  8]{PRAG}). Hence if
  $\sigma\in\prod_p\PGL_2(\Z_p)\backslash\PGL_2(\Q_p)_f$, then there
  exists an element $\gamma\in\PGL_2(\Q)$ such that $\gamma$ maps to
  $\sigma$ under the diagonal embedding. Since $\gamma\cdot f\in
  V_{\Z_p}$ for all $p$, we see that $\gamma\cdot f\in V_\Z$, implying
  $\gamma\in\PGL_2(\Q)_f$. Thus $\tau$ is surjective, completing the
  proof of the proposition.
\end{proof}

Thus the global weights of elements in $S(F)$ are products of local
weights, and so we may express the global weighted density of the set 
$S(F)$ in $V_\Z$ as a product of local weighted densities of the
closures of $S(F)$ in $V_{\Z_p}$.  
We compute these local
densities next, in terms of local masses of 2-coverings of elliptic
curves. 

\subsection{Local densities of the weighted set
$S(F)$ in terms of local masses of 2-coverings of elliptic curves in $F$}

Let $F$ be a large family of elliptic curves. 
Let $S(F)$ again denote the set of all locally soluble integral binary
quartic forms having invariants $2^4I$ and $2^6J$ where
$(I,J)\in \Inv(F)$, and let $S_p(F)$ denote the
$p$-adic closure of $S(F)$ in $V_{\Z_p}$. We now determine the
$p$-adic density of $S_p(F)$, where each element $f\in
S_p(F)$ is weighted by $1/m_p(f)$,
in terms of a {\it local $(p$-adic$)$ mass} $M_p(V,F)$ involving all
isomorphism classes of soluble $2$-coverings of elliptic curves over
$\Q_p$ whose invariants lie in $\Inv_p(F)$.
To do so we need the following proposition, which is a reformulation of 
the change-of-measure assertion of Proposition~\ref{bqjac} with
$\Z_p$ in place of $\R$; we postpone the proof to \S\ref{secjac}.

\begin{proposition}\label{jacpgl2}
Let $p$ be a prime, and let $\phi$ be a continuous function
  on $V_{\Z_p}$. Then
  \begin{equation}\label{eqjacpgl2}
    \int_{V_{\Z_p}}\phi(f)df=\Bigl|\frac{1}{27}\Bigr|_p
\int_{\substack{(I,J)\in \Z_p^2\\\Delta(I,J)\neq 0}}\Bigl(\sum_{f\in\textstyle{\frac{V_{\Z_p}(I,J)}{\PGL_2(\Z_p)}}}\frac{1}{\#\Aut_{\Z_p}(f)}\int_{g\in\PGL_2(\Z_p)}\phi(g\cdot f)\omega(g)\Bigr)dIdJ,
  \end{equation}
  where $\frac{V_{\Z_p}(I,J)}{\PGL_2(\Z_p)}$ denotes a set of
  representatives for the action of $\PGL_2(\Z_p)$ on elements in
  $V_{\Z_p}$ having invariants $I$ and $J$.
\end{proposition}

In certain special cases where $\phi(f)$ is additionally weighted by $1/m_p(f)$, Equation (\ref{eqjacpgl2}) takes on a particularly nice form:
\begin{corollary}\label{corjac}
  Let $p$ be a prime and let $\phi$ be a continuous $\PGL_2(\Q_p)$-invariant
  function on $V_{\Z_p}$ such that every element $f\in
  V_{\Z_p}$ in the support of $\phi$ has nonzero discriminant, is soluble, and satisfies $2^4\cdot 3\mid
  I(f)$ and $2^6\cdot 3^3\mid J(f)$. Then
  \begin{equation}\label{eqjacpgl3}
    \int_{V_{\Z_p}}\frac{\phi(f)}{m_p(f)}df=\Bigl|\frac{1}{27}\Bigr|_p\Vol(\PGL_2(\Z_p))
\int_{\substack{(I,J)\in \Z_p^2\\\Delta(I,J)\neq 0}}\frac{1}{\# E[2](\Q_p)}\Bigl(\sum_{\sigma\in E(\Q_p)/2E(\Q_p)}\phi(f_\sigma)\Bigr)dIdJ,
  \end{equation}
  where $f_\sigma$ is any element in $V_{\Z_p}$ that corresponds to
  $\sigma$ under the correspondence of Theorem~$\ref{thparame2e}$. $($The
  existence of such an $f_\sigma\in V_{\Z_p}$ is the content of
  Lemma~$\ref{lemreduction}.)$
\end{corollary}
\begin{proof}
  Proposition \ref{jacpgl2} implies that we have
  \begin{equation}\label{eqjaccorf}
    \begin{array}{rcl}
\displaystyle\int_{V_{\Z_p}}\displaystyle\frac{\phi(f)}{m_p(f)}df&=&\Bigl|\displaystyle\frac{1}{27}\Bigr|_p
\displaystyle\int_{\substack{(I,J)\in \Z_p^2\\\Delta(I,J)\neq 0}}\Bigl(\displaystyle\sum_{f\in\textstyle{\frac{V_{\Z_p}(I,J)}{\PGL_2(\Z_p)}}}\displaystyle\frac{1}{\#\Aut_{\Z_p}(f)}\int_{g\in\PGL_2(\Z_p)}\displaystyle\frac{\phi(g\cdot f)}{m_p(g\cdot f)}dg\Bigr)dIdJ\\[0.3in]
&=&\Bigl|\displaystyle\frac{1}{27}\Bigr|_p\Vol(\PGL_2(\Z_p))
\displaystyle\int_{\substack{(I,J)\in \Z_p^2\\\Delta(I,J)\neq 0}}\Bigl(\displaystyle\sum_{f\in\textstyle{\frac{V_{\Z_p}(I,J)}{\PGL_2(\Z_p)}}}\displaystyle\frac{\phi(f)}{m_p(f)\#\Aut_{\Z_p}(f)}\Bigr)dIdJ,
    \end{array}
  \end{equation}
  since both $\phi$ and $m_p$ are
  $\PGL_2(\Z_p)$-invariant. We now evaluate the sum within the
  integral in the second line of (\ref{eqjaccorf}). For $f\in
  V_{\Z_p}$, let $f=f_1,f_2,\ldots,f_k$ be the set of all elements in
  $\frac{V_{\Z_p}(I,J)}{\PGL_2(\Z_p)}$ that are
  $\PGL_2(\Q_p)$-equivalent to $f$. Then since $\phi$ and $m_p$ are $\PGL_2(\Q_p)$-invariant, we have
\begin{equation*}
  \begin{array}{rcl}    
\displaystyle\sum_{i=1}^k\displaystyle\frac{\phi(f_i)}{m_p(f_i)\#\Aut_{\Z_p}(f_i)}&=&\displaystyle\frac{\phi(f)}{m_p(f)}\displaystyle\sum_{i=1}^k\displaystyle\frac{1}{\#\Aut_{\Z_p}(f_i)}\;\;=\;\;
\phi(f)\left(\displaystyle\sum_{i=1}^k\displaystyle\frac{\#\Aut_{\Q_p}(f)}{\#\Aut_{\Z_p}(f_i)}\right)^{-1}\displaystyle\sum_{i=1}^k\displaystyle\frac{1}{\#\Aut_{\Z_p}(f_i)}\\[0.2in] &=&\displaystyle\frac{\phi(f)}{\#\Aut_{\Q_p}(f)}.
  \end{array}
\end{equation*}
Therefore, we obtain
\begin{equation}\label{eqjaccors}
\int_{V_{\Z_p}}\frac{\phi(f)}{m_p(f)}df=
\Bigl|\frac{1}{27}\Bigr|_p\Vol(\PGL_2(\Z_p))\int_{\substack{(I,J)\in \Z_p^2\\\Delta(I,J)\neq 0}}\Bigl(\displaystyle\sum_{f\in\textstyle{\frac{V_{\Z_p}(I,J)}{\PGL_2(\Q_p)}}}\displaystyle\frac{\phi(f)}{\#\Aut_{\Q_p}(f)}\Bigr)dIdJ,  
\end{equation}
where $\frac{V_{\Z_p}(I,J)}{\PGL_2(\Q_p)}$ analogously denotes a set
consisting of one element from each $\PGL_2(\Q)$-equivalence class in
$V_{\Z_p}$ having invariants $I$ and $J$. Theorem \ref{thparame2e} and
Lemma \ref{lemreduction} imply that soluble elements in
$\frac{V_{\Z_p}(I,J)}{\PGL_2(\Q_p)}$ are in bijective correspondence
with elements in $E(\Q_p)/2E(\Q_p)$. Theorem \ref{thparame2e} further
states that $\Aut_{\Q_p}(f)$ is isomorphic to
$E^{I(f),J(f)}[2](\Q_p)$. Therefore, Corollary \ref{corjac} follows
from (\ref{eqjaccors}).
\end{proof}

We now have the following proposition which determines the necessary
local $p$-adic masses.
\begin{proposition}\label{denel}
We have
$$\int_{S_p(F)}\frac{1}{m_p(f)}df=|2^{10}/27|_p\cdot\Vol(\PGL_2(\Z_p))\cdot M_p(V,F),$$
where
\begin{equation}\label{eqmpvf}
M_p(V,F):=\displaystyle\int_{(I,J)\in \Inv_p(F)}\frac{\#(E^{I,J}(\Q_p)/2E^{I,J}(\Q_p))}{\#E^{I,J}(\Q_p)[2]}dIdJ.
\end{equation}
\end{proposition}
\begin{proof}
  The set $S_p(F)$ consists of all $\Q_p$-soluble binary quartic forms
  having invariants $2^4I$ and $2^6J$ with $(I,J)\in
  \Inv_p(F)$. Proposition \ref{denel} thus follows directly from
  Corollary \ref{corjac} since $E^{I,J}(\Q_p)$ is isomorphic to
  $E^{2^4I,2^6J}(\Q_p)$ and the volume of $\{(2^4I,2^6J)|(I,J)\in
  \Inv_p(F)=|2^{10}|_p\cdot\Vol(\Inv_p(F))$.
\end{proof}

\subsection{A change-of-measure formula}\label{secjac}

In this subsection, our aim is to prove the change-of-variables formula
that is contained in Proposition~\ref{bqjac} and
Proposition~\ref{jacpgl2} over $\R$ and over $\Q_p$, respectively.
We begin by proving first the following result over $\C$:
\begin{proposition}\label{propjacfirststep}
  Let $\omega$, $dv$, and $dIdJ$ be as in Proposition~$\ref{bqjac}$. Let $R\subset\C^2$ be an open set and $s:R\to V_\C$
  be a continuous function such that the binary quartic form
  $s_{I,J}:=s(I,J)$ has invariants equal to $I$ and $J$ for each
  $(I,J)\in R$. Then there exists a nonzero rational number $\J$ such
  that for any measurable function $\phi:V_\C\to\R$, we have
$$\int_{v \in \PGL_2(\C)\cdot
  s(R)}\phi(v)dv=|\J|\int_{R}\int_{\PGL_2(\C)}
\phi(g\cdot s_{I,J})\,\omega(g) \,dIdJ,$$
where we regard $\PGL_2(\C)\cdot s(R)$ as a multiset.
\end{proposition}
\begin{proof}
  Let us begin with the special case when the function $s$ is locally 
analytic. Then we know that
\begin{equation}
  \label{eqjacfirststep}
\int_{v \in \PGL_2(\C)\cdot s(R)}\phi(v)dv=\int_{(I,J)\in\C^2}\int_{\PGL_2(\C)}\J_s(g,I,J)\phi(g\cdot s_{I,J})\,\omega(g) \,dIdJ,  
\end{equation}
where $\J_s(g,I,J)$ is the Jacobian change of variables of the map 
\begin{equation}
  \begin{array}{rcl}
    \psi_s:\PGL_2(\C)\times R&\to& V_\C\\
    (g,(I,J))&\mapsto&g\cdot s_{I,J}.
  \end{array}
\end{equation}
Note that $\J_s(g,I,J)$ is continuous in $g$, $I$, and $J$.
In what follows, we prove that $\J_s(g,I,J)$ is independent of $g$, $I$, $J$, and $s$.

\vspace{.1in}
\noindent{\bf Step 1:} $\J_s(g,I,J)$ is independent of $g\in\PGL_2(\C)$.

\vspace{.065in}
Suppose there exists $(I,J)\in R$ and $g_1,g_2\in\PGL_2(\C)$ such that
$\J_s(g_1,I,J)\neq\J_s(g_2,I,J)$. Then, by continuity and the fact that $\omega(g)$ is $\PGL_2(\C)$-invariant, there exists an open set
$B_1\subset\PGL_2(\C)$ containing $g_1$ such that
$\int_{B_1}\J_s(g,I,J)\omega(g)\neq\int_{g_2g_1^{-1}B_1}\J_s(g,I,J)\omega(g)$. By continuity, there then exists an open set $N\subset R$ containing $(I,J)$ such that
\begin{equation}
  \label{eqindgjac}
\int_{(I,J)\in
  N}\int_{B_1}\J_s(g,I,J)\omega(g)dIdJ\neq\int_{(I,J)\in
  N}\int_{g_2g_1^{-1}B_1}\J_s(g,I,J)\omega(g)dIdJ.  
\end{equation}
From (\ref{eqjacfirststep}) it follows that the left hand side of
(\ref{eqindgjac}) is equal to the volume of $B_1\cdot N$ while the
right hand side of (\ref{eqindgjac}) is equal to the volume of
$g_2g_1^{-1}B_1\cdot N$. Since the map $g_2g_1^{-1}:V_\C\to V_\C$ is
via an element in $\SL(V_\C)$, we obtain the desired contradiction.

\vspace{.1in}
\noindent{\bf Step 2:} $\J_s(I,J):=\J_s(g,I,J)$ is independent of $s$.

\vspace{.065in}
Let $s':R\to V_\C$ be another locally analytic function such that the
invariants of $s'_{I,J}:=s'(I,J)$ are $I$ and $J$ for each
$(I,J)\in R$. Since $\PGL_2(\C)\cdot s(R)$ and $\PGL_2(\C)\cdot s'(R)$ are the same multisets, we have
$$\int_{v \in \PGL_2(\C)\cdot s'(R)}\phi(v)dv=\int_{v \in \PGL_2(\C)\cdot s(R)}\phi(v)dv=\int_{(I,J)\in\C^2}\int_{\PGL_2(\C)}\J_s(I,J)\phi(g\cdot s_{I,J})\,\omega(g) \,dIdJ.$$
For each $(I,J)\in \C^2$ let $g_{I,J}\in\PGL_2(\C)$ be such that $g_{I,J}\cdot s_{I,J}=s'_{I,J}$. Then, because $\omega(g)$ is both a left and a right Haar-measure, we obtain
\begin{eqnarray*}
\int_{(I,J)\in\C^2}\int_{g\in \PGL_2(\C)}\J_s(I,J)\phi(g\cdot s_{I,J})\omega(g) dIdJ&=&\int_{\C^2}\int_{\PGL_2(\C)}\J_s(I,J)\phi(g g_{I,J}\cdot s_{I,J})\omega(g) dIdJ\\
&=&\int_{\C^2}\int_{\PGL_2(\C)}\J_s(g,I,J)\phi(g\cdot s'_{I,J})\omega(g) dIdJ.
\end{eqnarray*}
Hence it follows that
$$\int_{v \in \PGL_2(\C)\cdot s'(R)}\phi(v)dv=\int_{(I,J)\in\C^2}\int_{\PGL_2(\C)}\J_s(I,J)\phi(g\cdot s'_{I,J})\,\omega(g) \,dIdJ.$$
Thus $\J_{s'}(I,J)=\J_s(I,J)$ as desired.

\vspace{.1in}
\noindent{\bf Step 3:} $\J(I,J):=\J_s(I,J)$ is a nonzero polynomial in $I$ and $J$ with rational coefficients.

\vspace{.065in}
We can choose $s$ such that the coefficients of $s_{I,J}$ are rational
polynomials in $I$ and $J$; for example, let
$s_{I,J}:=x^3y-\frac{I}{3}xy^3-\frac{J}{27}y^4$. Since $\J(I,J)$ is
the determinant of a $5\times 5$ matrix whose entries are polynomials in the coefficients
of $s_{I,J}$, it follows that $\J(I,J)$ is a rational polynomial in
$I$ and $J$. Because $\psi_s(\PGL_2(\C),\C^2)$ is a full measure set
in $V_\C$, we obtain that $\J(I,J)$ is nonzero.

\vspace{.1in}
\noindent{\bf Step 4:} $\J:=\J(I,J)$ is a nonzero rational constant.

\vspace{.065in}
  Let $G_0\subset \PGL_2(\C)$ be a bounded subset having volume $1$ and let
  $R_0$  be any bounded measurable set in $\C^2$.
We denote the set of all elements $s_{I,J}$ with
  $(I,J)\in R_0$ by $B=B(R_0)$. Then 
\begin{equation}\label{scalingeq0}
\int_{G_0\cdot B}dv=\int_{(I,J)\in R_0}\J(I,J)dIdJ,
\end{equation}
where we view $G_0\cdot B$ as a multiset.  
Now for any $c\in \C$, we have by (\ref{scalingeq0}) that
\begin{equation}\label{scalingeq1}
\int_{cG_0\cdot B}dv=|c|^{5}\int_{G_0\cdot B}dv=|c|^5\int_{(I,J)\in R_0}\J(I,J)dIdJ
\end{equation}
  because $V_\C$ has dimension $5$.  On the other hand, we may
  evaluate the left hand side of (\ref{scalingeq1}) in another way; namely,
using (\ref{scalingeq0}) with $cB$ in place of $B$, we obtain
\begin{equation}\label{scalingeq2}
\int_{cG_0\cdot B}dv=\int_{G_0\cdot cB}dv=\int_{(c^{-2}I,c^{-3}J)\in
  R_0}\J(I,J)dIdJ =\int_{(I',J')\in
  R_0}\J(c^2I',c^3J')\,|c^2|\,dI'\,\,|c^3|dJ'
\end{equation}
because $I$ and $J$ are homogeneous polynomials of degree 2 and 3,
respectively. 
Comparing the
right hand sides of (\ref{scalingeq1}) and (\ref{scalingeq2}), we
obtain
\begin{equation}\label{scalingeq3}
\int_{(I,J)\in R_0}\J(I,J)dIdJ=\int_{(I,J)\in
  R_0}\J(c^2I,c^3J)dIdJ.
\end{equation}
Since, by Step 3, $\J(I,J)$ is a nonzero polynomial in $I$ and $J$ having
rational coefficients, and since the equality $(\ref{scalingeq3})$ is
true for all $R_0$ and all $c$, 
we conclude that $\J(I,J)$ must be a nonzero rational constant.

\vspace{0.1in}
Finally, as every continuous function can be locally uniformly approximated
as closely as desired by locally analytic functions (by the Stone--Weierstrass
theorem), the proposition follows.
\end{proof}
\noindent 

Proposition \ref{bqjac}, with $1/27$ replaced by $\J$, now follows
from Proposition \ref{propjacfirststep} and the principle of
permanence of identities. 
More generally, we have 
obtained
the following result:
\begin{proposition}\label{proptocite}
  Let $K$ be $\R$, $\C$, or $\Z_p$ for some prime $p$. Let $dv$ be the
  standard additive measure on $V_K$, the space of all binary quartic
  forms with coefficients in $K$. Let $R$ be an open subset of
  $K\times K$ and let $s:R\to V_K$ be a continuous function such that
  the invariants of $s_{I,J}:=s(I,J)$ are $I$ and $J$.  Then there
  exists a rational nonzero constant $\J$ such that for any measurable
  function $\phi$ on $V_K$, we have
\begin{equation}\label{eqjacpglj}
\int_{v \in \PGL_2(K)\cdot
  s(R)}\phi(v)dv=  |\J|\int_{R}\int_{\PGL_2(K)}
\phi(g\cdot s_{I,J})\,\omega(g) \,dIdJ,
\end{equation}
where we regard $\PGL_2(K)\cdot s(R)$ as a
multiset, $\omega$ is as defined in Section $2.4$, and $|\J|$ denotes
the usual absolute value of $\J$ as an element of $K$. 
\end{proposition}

\vspace{0.1in}

We next wish to prove the statement of Proposition \ref{jacpgl2},
with $1/27$ replaced by $\J$. To do this, because every continuous
function on $V_{\Z_p}$ is locally constant outside a set of
arbitrarily small measure, we may assume that $\phi$ is locally
constant. 
Also, it suffices to prove the statement locally; i.e., for
every element $f\in V_{\Z_p}$ (we may also assume that $\Delta(f)\neq
0$) there exists a neighborhood $B_f$ of $f$ such that
(\ref{eqjacpgl2}), with $1/27$ replaced by $\J$, is true when $\phi$
is the characteristic function of $B_f$. 

Given $f\in
V_{\Z_p}\backslash\{\Delta=0\}$, we now construct such a
neighborhood $B_f$. Let $P\subset V_{\Z_p}$ be a generic $2$-dimensional
plane passing through $f$ defined by linear equations over $\Q$; then
there exists a neighborhood $P_0\subset P$ of $f$ such that the
invariants of any two elements in $P_0$ are distinct in $\Z_p^2$ and
the size of the stabilizers in $\PGL_2(\Z_p)$ of any two elements in
$P_0$ are equal. The first claim in the previous statement follows
from the inverse function theorem for local fields (see
\cite[Proposition 4.3]{Sch}) used on the usual map from
$\PGL_2(\Z_p)\times P$ to $V_{\Z_p}$.  Then we define $B_f$ to be
$\PGL_2(\Z_p)\cdot P_0$ (regarded as a set, not a multiset). Since
the plane $P$ was defined by linear equations over $\Q$, Proposition
\ref{propjacfirststep} and the principle of permanence of identities
implies that
$$
\#\Aut_{\Z_p}(f)\cdot\Vol(B_f)=|\J|_p\cdot\Vol(\PGL_2(\Z_p))\cdot\int_{\Inv_p(P_0)}dIdJ,
$$
where $\Inv_p(P_0)$ denotes the set of all $(I,J)\in\Z_p^2$ that occur
as invariants of some element in $P_0$. We have thus proven
Proposition \ref{jacpgl2}, with $1/27$ replaced by $\J$. 
In fact, our argument yields the following result:

\begin{proposition}\label{jacpglcite}
  Let $K$ be $\R$, $\C$, or $\Z_p$ for some prime $p$, and let $\phi$
  be a measurable function on $V_{K}$. Then there exists a rational
  constant $\J$, independent of $K$ and $\phi$, such that
  \begin{equation}
    \int_{V_{K}}\phi(f)df=|\J|
\int_{\substack{(I,J)\in K^2\\\Delta(I,J)\neq 0}}\Bigl(\sum_{f\in\textstyle{\frac{V_{K}(I,J)}{\PGL_2(K)}}}\frac{1}{\#\Aut_{K}(f)}\int_{g\in\PGL_2(K)}\phi(g\cdot f)\omega(g)\Bigr)dIdJ,
  \end{equation}
  where $\frac{V_{K}(I,J)}{\PGL_2(K)}$ denotes a set of
  representatives for the action of $\PGL_2(K)$ on elements in
  $V_{K}$ having invariants $I$ and $J$.
\end{proposition}

To complete the proof of Proposition \ref{jacpgl2}, it only
remains to show that the absolute value of $\J$ is equal to $1/27$. We
accomplish this by computing the value of $|\J|_p$ for each prime~$p$.  
Namely, for each prime~$p$, we pick an appropriate set $S\subset
V_{\Z_p}$, and then use (\ref{eqjacpglj}) to express $|\J|_p$ in terms of
the volume of $S$. We then consider $\bar{S}$, the reduction of $S$
modulo $p$, and determine its cardinality to explicitly compute the
volume of $S$, and thereby determine the value of $|\J|_p$.

To this end, we have the following proposition.
\begin{proposition}\label{propjacvol1}
  Let $p$ be a fixed prime number. Let $S\subset
  V_{\Z_p}$ be a set defined by congruence conditions modulo $p$, and let $\bar{S}\subset
  V_{\F_p}$ denote the reduction of $S$ modulo $p$. Assume that $S=\pi^{-1}(\pi(S))$, where $\pi$ is given by taking invariants. Then
\begin{equation}\label{eqcompjfirst}
|\J|_p=
\frac{\#\PGL_2(\F_p)\cdot\Bigl(\displaystyle\sum_{f\in\PGL_2(\F_p)\backslash\bar{S}}\frac1{\#\Aut_{\F_p}(f)}\Bigr)}
{p^{\dim V}\cdot\Vol(\PGL_2(\Z_p))\cdot\Bigl(\displaystyle\int_{(I,J)\in\pi(S)}\displaystyle\sum_{f\in\textstyle{\frac{V_{\Z_p}(I,J)}{\PGL_2(\Z_p)}}}\frac{1}{\#\Aut_{\Z_p}(f)}dIdJ\Bigr)}.
\end{equation}
\end{proposition}
\begin{proof}
Using Proposition \ref{jacpglcite} with $\phi$ replaced by the characteristic function of $S$, we obtain
  \begin{equation}\label{eqpfone}
    \Vol(S)=|\J|_p
\Vol(\PGL_2(\Z_p))\displaystyle\int_{(I,J)\in\pi(S)}\Bigl(\displaystyle\sum_{f\in\textstyle{\frac{V_{\Z_p}(I,J)}{\PGL_2(\Z_p)}}}\frac{1}{\#\Aut_{\Z_p}(f)}\Bigr)dIdJ.
  \end{equation}
Since $S$ is defined by congruence conditions modulo $p$, and since
$\bar{S}$ is $\PGL_2(\F_p)$ invariant (a consequence of the
$\PGL_2(\Z_p)$-invariance of $S$), we have
\begin{equation}\label{eqpftwo}
\Vol(S)=\displaystyle\frac{\#\bar{S}}{p^{\dim V}}=\frac1{p^{\dim V}}\#\PGL_2(\F_p)\cdot\Bigl(\displaystyle\sum_{f\in\PGL_2(\F_p)\backslash\bar{S}}\frac1{\#\Aut_{\F_p}(f)}\Bigr),
\end{equation}
where the final equality follows from the orbit-stabilizer
formula. Equating the right hand sides of (\ref{eqpfone}) and
(\ref{eqpftwo}) yields the proposition.
\end{proof}

\begin{remark}{\em Thus far, we have not used anything specific about
binary quartic forms, and the analogues of the statements and proofs of Propositions \ref{proptocite}---\ref{propjacvol1}
continue to hold if we replace the pair $(\PGL_2,V)$ with any representation $(G,W)$ defined over $\Z$, as
long as the following conditions hold:
\begin{enumerate}
\item $G$ is a semisimple group and $W$ is a {\it coregular} representation of $G$, i.e., the ring of invariants for the action of $G_\C$ on $W_\C$ is freely generated, say, by the polynomials $I_1,\ldots,I_k$ (which we may take to be integral polynomials).
\item The stabilizer in $G_\C$ of any element $v\in W_\C$ outside a
  measure $0$ set of $W_\C$ is finite and absolutely bounded.
\item The sum of the degrees of the $I_j$'s is equal to the dimension of $W$ (in the case of binary quartic forms, we
  had $2+3=5$). This condition is necessary to prove that the relevant
  Jacobian change of variables $\mathcal J$ is independent of the values of $I_1,\ldots,I_k$ in Step 4.
\item There exists a rational polynomial map $\phi:\C^k\to W_\C$ such that $\phi(i_1,\ldots,i_k)$ has invariants $(i_1,\ldots,i_k)$ for each $k$-tuple in $\C^k$.
\end{enumerate}}
\end{remark}

In our case of binary quartic forms, to apply Proposition~\ref{propjacvol1} we may choose $S$, e.g.,  to be the set of binary quartic forms in $V_{\Z_p}$ having some fixed invariants $(I,J)$ modulo $p$.  
The following lemma is then useful in evaluating the right hand side of (\ref{eqcompjfirst}).
\begin{lemma}\label{lemmaprop}
  Let $p$ be a fixed prime, and let $(I,J)\in\Z_p^2$ be an element in
  the image of $\pi$ such that $p^2\nmid\Delta(I,J)$. Then
$$
\displaystyle\sum_{f\in\textstyle\frac{V_{\Z_p}(I,J)}{\PGL_2(\Z_p)}}\frac1{\#\Aut_{\Z_p}(f)}=1.
$$
Let $p\neq 3$ be a prime, and let $(I,J)\in\F_p^2$ be an element such that $\Delta(I,J)\neq 0$. Then
$$
\displaystyle\sum_{f\in\textstyle\frac{V_{\F_p}(I,J)}{\PGL_2(\F_p)}}\frac1{\#\Aut_{\F_p}(f)}=1.
$$
\end{lemma}
\begin{proof}
    Since $p^2\nmid\Delta(I,J)$, Theorem \ref{thparame2e} and Proposition \ref{unifellip} imply that
    \begin{equation}\label{eqlemprop1}
    \Aut_{\Z_p}(f)=\Aut_{\Q_p}(f)=E^{I,J}(\Q_p)[2].
    \end{equation}
    For odd primes $p$, Theorem \ref{thparame2e} and
    \cite[Lemmas 3, 4]{BSD} show that the number of
    $\PGL_2(\Q_p)$-equivalence class in $V_{\Z_p}$ having invariants
    $I$ and $J$ is equal to $\#(E^{I,J}(\Q_p)/2E^{I,J}(\Q_p))$, while
    the results in \cite[Section 6]{CS} show that the number of
    $\PGL_2(\Q_2)$-equivalence class in $V_{\Z_2}$ having invariants
    $I$ and $J$ is equal to $\frac12\#(E^{I,J}(\Q_2)/2E^{I,J}(\Q_2))$.
    The first assertion of Lemma \ref{lemmaprop} now follows from
    Lemma~\ref{bk1}, which states that the value of
    $\#(E^{I,J}(\Q_p)/2E^{I,J}(\Q_p))/\#E^{I,J}(\Q_p)[2]$ is $1$ if
    $p\neq 2$, and $2$ if $p=2$.

    For $p\geq 5$, the second assertion of Lemma \ref{lemmaprop}
    follows from Theorem \ref{thparame2e} with $K$ replaced by $\F_p$,
    and the fact that
    $\#(E^{I,J}(\F_p)/2E^{I,J}(\F_p))/\#E^{I,J}(\F_p)[2]$ is $1$.
    For $p=2$, the lemma follows from a finite computation.
\end{proof}

Let us now choose some specific sets $S\subset V_{\Z_p}$ for each prime $p$. If $p\neq 3$,
let $(I_0,J_0)\in\F_p^2$ be a fixed element such that
$\Delta(I_0,J_0)\neq 0$.
We then define $S$ to be the set of
all $f\in V_{\Z_p}$ such that the reduction of $(I(f),J(f))$ modulo $p$ is equal to
$(I_0,J_0)$. Then Proposition \ref{propjacvol1} in conjunction with Lemma \ref{lemmaprop} implies that
$$
|\J|_p=\frac{\#\PGL_2(\F_p)}{p^5\Vol(\PGL_2(\Z_p))(1/p^2)}=1.
$$

Because the definition of $\Delta$ in terms of $I$ and $J$ requires
division by 27, specifying a given value of $(I,J)$ modulo 3 cannot
alone guarantee that $3\nmid \Delta(I,J)$ (this is indeed the reason
for excluding the case $p=3$ in Lemma~\ref{lemmaprop}).  Hence, in the
case $p=3$, we choose instead a set $S$ defined by conditions on the
invariants $(I,J)$ modulo a higher power of 3.  For example, let $S$
be the set of all $f\in V_{\Z_3}$ such that $I(f)\equiv 3\pmod{9}$.
The proof of Theorem \ref{eligible} immediately implies that if $f\in
V_{\Z_3}$ and $I(f)\equiv 0\pmod{3}$, then the only condition on $J$
is that $J(f)\equiv 0\pmod{27}$.  Thus, if
$f(x,y)=ax^4+bx^3y+cx^2y^2+dxy^3+ey^4\in S$, then $\Delta(f)\not\equiv
0\pmod3$, and we may use the first statement of Lemma \ref{lemmaprop}.
Next, note that $I(f)\equiv 3\pmod{9}$ precisely when $c\equiv
0\pmod{3}$ and $ae-bd\equiv 1\pmod{3}$. Let
$\bar{a},\;\bar{b},\;\bar{c},\;\bar{d}$, and $\bar{e}$ denote the
reductions modulo $3$ of $a,\;b,\;c,\;d$, and $e$, respectively. Then
$f\in S$ if and only if $\bar{c}=0$ and
$\bar{a}\bar{e}-\bar{b}\bar{d}=1$.  There are $24$ values of
$(\bar{a},\bar{b},\bar{c},\bar{d},\bar{e})\in\F_p^5$ satisfying these
two conditions. Therefore,
  $$
|\J|_3=\frac{24}{3^5\Vol(\PGL_2(\Z_p))\Vol(\pi(S))}=\frac{24}{3^5(1-1/3^2)(1/3^5)}=27.
$$
This completes the proof of Proposition \ref{jacpgl2}.

Alternatively, we could choose $S$ to be the set of $f\in V_{\Z_p}$ such
that $p\nmid\Delta(f)$. Then $\bar{S}$, the reduction of $S$ modulo
$p$, is the set of all $f\in V_{\F_p}$ such that $\Delta(f)\neq 0$. An
element of $\bar{S}$ is determined, up to scaling by elements in
$\F_p^\times$, by its roots in
$\P^1_{\overline{\F}_p}$. For example, the number of elements in
$\bar{S}$ having four distinct roots in $\P^1(\F_p)$ is
$(p-1)\frac{1}{24}(p+1)p(p-1)(p-2)$. An elementary computation then yields the following equality:
$$
\#\bar{S}=p^2(p+1)(p-1)^2.
$$
Therefore, \eqref{eqpfone} and Lemma \ref{lemmaprop} imply that we have
$$
|\J|_p=\frac{\Vol(S)}{\Vol(\PGL_2(\Z_p))\Vol(\pi(S))}=\frac{\#\bar{S}}{p^5\Vol(\PGL_2(\Z_p))\Vol(\pi(S))}=\frac{p-1}{p\Vol(\pi(S))}.
$$
The set $\pi(S)$ consists of eligible pairs $(I,J)\in\Z_p^2$ such that
$p\nmid\Delta(I,J)$. (A pair $(I,J)\in\Z_p^2$ is said to be {\it
  eligible} if it occurs as the invariants of some $f\in V_{\Z_p}$.)
We may thus use Theorem \ref{eligible} and compute the volume of
$\pi(S)$ to be $(p-1)/p$ when $p\neq 3$ and $2/81$ when $p=3$. We thus
again obtain $|\J|_p=1$ for $p\neq 3$ and $|\J|_3=27$, 
yielding Proposition \ref{jacpgl2}.

\subsection{The number of elliptic curves of bounded height in a large family}
Suppose $F$ is a large family of elliptic curves. To prove
Theorem \ref{ellipall} we need to estimate
the number of elliptic curves in $F$ that have height bounded by
$X$. In this section, we determine exact asymptotics for the number of
elliptic curves having bounded height in any large family $F$ of
elliptic curves.

As an elliptic curve is determined by its invariants $I$ and $J$, we
estimate the number of pairs $(I,J)$ that belong to $\Inv(F)$ and have
height less than $X$. It follows from an easy application of
Proposition \ref{davlem} that the number of pairs $(I,J)\in\Z\times\Z$
satisfying $H(I,J)<X$ and $4I^3-J^2>0$ (resp.\ $H(I,J)<X$ and
$4I^3-J^2<0$) is equal to the volume of $R^+_X$ (resp.\ $R^-_X$) up to
an error of $O(X^{1/2})$, where the sets $R^\pm_X$ were defined in the
proof of Proposition \ref{IJcount}.  For any set $S\subset
\Z\times\Z$, let $N(S;X)$ denote the number of pairs $(I,J)\in S$,
having height bounded by $X$, satisfying $\Delta(I,J)\neq 0$.

Now, the set $\Inv(F)\subset\Z\times\Z$ is defined by (perhaps
infinitely many) congruence conditions.  To determine the asymptotics
of $N(\Inv(F);X)$ as $X$ goes to infinity, we need the following
uniformity estimate:
\begin{proposition}\label{propunifec}
  The number of elliptic curves $E$ over $\Q$ having height less than
  $X$ such that $p^2$ divides the discriminant of $E$ is
  $O(X^{5/6}/p^{3/2})$, where the implied constant is independent of
  $p$.
\end{proposition}
\begin{proof}
  This proof is very similar to (but much easier than) the proof of
  the uniformity estimate for binary quartic forms in Theorem
  \ref{thunifbqelem}. We start with embedding the
  set $\{x^3+Ax+B:A,B\in\Z\}$ into the bigger space of all integral
  binary cubic forms. Let $U_\Z$ denote the space of all integral
  binary cubic forms. The group $\GL_2(\Z)$ acts on $U_\Z$ by linear substitution of variables. Consider the composite map $\psi=\psi_2\circ\psi_1$ given by
$$
\psi:\{x^3+Ax+B:A,B\in\Z\}\to U_\Z\to\GL_2(\Z)\backslash U_\Z,
$$
where the first map $\psi_1$ sends $x^3+Ax+B$ to the integral binary
cubic form $x^3+Axy^2+By^3$. As in the proof of Proposition
\ref{atmosttwelve}, an element in $\GL_2(\Z)\backslash U_\Z$ has at
most $12$ preimages under $\psi$. This can be seen as follows: if $f$
is in the preimage of the $\GL_2(\Z)$-orbit of $v\in U_\Z$, then there
exists an element $\gamma\in\GL_2(\Z)$ such that $\gamma\cdot
v=\psi_1(f)$. Then $v((1,0)\cdot\gamma)=1$ since $\psi_1(f)$ has
$x^3$-coefficient equal to $1$. The results in \cite{Del} and
\cite{Ev} assert that there are at most $12$ solutions
$(a,b)\in\Z\times\Z$ to the equation $v(a,b)=1$. This implies that $v$
has at most $12$ preimages under $\psi$ because each preimage yields a
different solution to $v(a,b)=1$. From \cite[Proposition 1]{DH}, it follows that the
number of $\GL_2(\Z)$-orbits on $U_\Z$ having discriminant divisible
by $p^2$ is bounded by $O(X/p^2)$. Therefore, the number of elliptic
curves having discriminant divisible $p^2$ is bounded by $O(X/p^2)$ as
well.

To complete the proof of the above proposition, we partition the set
of elliptic curves having discriminant divisible by $p^2$ into two
subsets.  First, consider elliptic curves $E_{A,B}:y^2=x^3+Ax+B$ having
additive reduction at a prime $p>3$. This happens if and only if
$p\mid A$ and $p\mid B$. The number of such pairs $(A,B)\in\Z\times\Z$
having height less than $X$ is clearly bounded by
$O(X^{5/6}/p^2+X^{1/2}/p+1)$. Therefore, the number of elliptic curves
having additive reduction at $p$ and height less than $X$ is bounded
both by $O(X/p^2)$ and by $O(X^{5/6}/p^2+X^{1/2}/p+1)$. These combined
estimates yield a bound of $O(X^{5/6}/p^{5/3})$ which is sufficient.

Now consider those elliptic curves $E_{A,B}$ such that
$p^2\mid\Delta(E_{A,B})$, $E_{A,B}$ has multiplicative reduction at
$p$, and $H'(E_{A,B})<X$. Assuming that $p>3$, we now have $p\nmid
A$. Since $E_{A,B}$ has height bounded by $X$, there are $O(X^{1/3})$
possible choices for $A$ and $O(X^{1/2})$ possible choices for
$B$. With $A$ fixed, there are then $O(1)$ possible choices for the
reduction of $B$ modulo $p^2$. Therefore, the number of such elliptic
curves is bounded by $O(X^{1/3}\cdot(X^{1/2}/p^2+1))$. Combined with
the previously obtained bound of $O(X/p^2)$, we see that the number of
such elliptic curves $E_{A,B}$ is bounded by
$O(X^{5/6}/p^{3/2})$. This concludes the proof.
\end{proof}

Analogously to $M_p(V,F)$, we define the local mass $M_p(F)$ by
\begin{equation}\label{mpufdef}
M_p(F)=\int_{(I,J)\in \Inv_p(F)}dIdJ.
\end{equation}
We also define the following analogues at infinity of $M_p(F)$ and $M_p(V,F)$, respectively.
\begin{equation}\label{eqminf}
  \begin{array}{rcl}
    M_\infty(F;X)&:=&\displaystyle\int_{\substack{(I,J)\in\Inv_\infty(F)\\H(I,J)<X}}dIdJ,\\[.3in]
    M_\infty(V,F;X)&:=&\displaystyle\int_{\substack{(I,J)\in\Inv_\infty(F)\\H(I,J)<X}}\displaystyle\frac{\#(E^{I,J}(\R)/2E^{I,J}(\R))}{\# E^{I,J}(\R)[2]}dIdJ.
  \end{array}
\end{equation}

We now have the following theorem, which follows from Proposition \ref{propunifec} just as Theorem \ref{thsquarefreebq} followed from Theorem \ref{thunifbqelem}:
\begin{theorem}\label{thnumelip}
Let $F$ be a large family of elliptic curves and let $N(F;X)$
denote the number of elliptic curves $E\in F$ such that $H'(E)<X$. Then
\begin{equation}\label{eqnumelip}
N(F;X)=M_\infty(F;X)\prod_pM_p(F)+o(X^{5/6}).
\end{equation}
\end{theorem}

\subsection{Proofs of the main theorems (Theorems \ref{mainellip}, \ref{ellipcong}, and \ref{ellipall})}

Let us say that an element $f\in V_\Z$ is {\it bad at $p$} if either
$f$ is not $\Q_p$-soluble or $m_p(f)\neq 1$.
To deduce Theorem~\ref{ellipall}
from Theorem \ref{thsquarefreebq}, we need the following result:

\begin{proposition}\label{unifellip}
If an integral binary quartic form $f$ is bad at a prime $p>2$, then 
$p^2\mid\Delta(f)$.
\end{proposition}

\begin{proof}
  If $m_p(f)\neq 1$, then there exists
  $\gamma\in\PGL_2(\Q_p)\backslash\PGL_2(\Z_p)$ such that $\gamma\cdot
  f\in V_{\Z_p}$.  By replacing $f$ with a $\PGL_2(\Z_p)$-translate if
  necessary, we may assume that $\gamma=\bigl(\begin{smallmatrix}p^a &
    {}\\{}&p^{b}\end{smallmatrix}\bigr)$, with $a>b=0$. It then
  follows that the $x^4$-coefficient of $f$ is divisible by $p^2$ and
  the $x^3y$-coefficient of $f$ is divisible by~$p$, implying that
  $p^2\mid\Delta(f)$.

  We now show that if $f\in V_\Z$ is not $\Q_p$-soluble, then $f$ has
  splitting type $(1^21^2)$, $(2^2)$, or $(1^4)$ at~$p$, implying that
  $p^2\mid\Delta(f)$. First, if the discriminant of
  $f\in V_{\Z_p}$ is prime to $p$, then $f$ is $\Q_p$-soluble (see
  \cite[Chapter 3.6]{Cbook}). Also, if the splitting type of $f$ at
  $p$ is $(1^211)$ or $(1^31)$, then the reduction of $f$ modulo $p$
  has a simple root in $\P^1(\F_p)$, which then lifts to a root in
  $\P^1(\Q_p)$ by Hensel's Lemma. Thus $f$ is $\Q_p$-soluble.

  It remains to prove that if the splitting type of $f$ at $p$ is
  $(1^22)$, then $f$ is $\Q_p$-soluble. If $f\in V_{\Z_p}$ has
  splitting type $(1^22)$, then the reduction of $f$ modulo $p$ can be
  assumed to be of the form $\bar{a}x^2(x^2-\bar{n}y^2)$, where
  $\bar{n}$ is a nonresidue modulo $p$.  Hence we may assume that
  $f=a(x^2-kpy^2)(x^2-ny^2)$, where $a,n,k\in\Z_p$, the element
  $n\in\Z_p$ is a nonresidue when reduced modulo $p$, and $p\nmid
  a$. If $a$ is a square in $\Q_p$, then $f(1,0)$ is a square in
  $\Q_p$ and we are done. So we may assume that $a$ is a
  nonsquare. Now if $p\nmid x_0$, then $x_0^2-kp$ is a square in
  $\Q_p$; so it suffices to prove the existence of $\bar{x}_0\in
  \F_p^\times$ such that $\bar{x}_0^2-\bar{n}$ is a quadratic
  nonresidue modulo $p$.  Consider the first quadratic residue
  $\bar{x}_0^2=(c+1)\bar{n}$ appearing in the sequence
  $\bar{n},\,2\bar{n},\ldots,(p-1)\bar{n}$. Then
  $\bar{x}_0^2-\bar{n}=(c+1)\bar{n}-\bar{n}=c\bar{n}$ is a nonresidue, as
  was desired.
\end{proof}

Analogously to the sets $S_p(F)$, we define $S_\infty(F)$ to be the
set of all $\R$-soluble binary quartic forms in $V_\R$ whose
invariants belong to $\Inv_\infty(F)$.  Since
$\#(E^{I,J}(\R)/2E^{I,J}(\R))/\# E^{I,J}(\R)[2]$ is always equal to
$1/2$, the computation of the volume of the sets $\RR_X(L^{(i)})$ in
Section 2.4 and the definition of $M_\infty(F;X)$ implies that
$$
N(V_\Z\cap S_\infty(F);X)=\frac1{27} \Vol(\PGL_2(\Z)\backslash\PGL_2(\R)) M_\infty(V,F;X)+O(X^{3/4+\epsilon}).
$$

We now prove the following theorem, from which Theorem \ref{ellipall} will
be seen to follow.
\begin{theorem}\label{main}
Let $F$ be a large family of elliptic curves. Then we have
\begin{equation}\label{eqthsec5}
  \displaystyle\lim_{X\to\infty}\frac{\displaystyle\sum_{\substack{E\in F\\H'(E)<X}}(\#S_2(E)-1)}{\displaystyle\sum_{\substack{E\in F\\H'(E)<X}}1}=
\Vol(\PGL_2(\Z)\backslash\PGL_2(\R))\frac{M_\infty(V,F;X)}{M_\infty(F;X)} \displaystyle\prod_p\left[\Vol(\PGL_2(\Z_p))\frac{M_p(V,F)}{M_p(F)}\right].
\end{equation}
\end{theorem}
\begin{proof}
Note that by Theorem~\ref{2spar},  
the numerator of the left hand side of (\ref{eqthsec5}) 
is equal to the number
of locally soluble $\PGL_2(\Z)$-orbits on $S(F^{\inv})$ having
height bounded by $2^{12}X$ and no rational linear factor, 
where each orbit $\PGL_2(\Z)\cdot f$ is
counted with weight $1/m(f)$.
Thus, by Theorem \ref{thsquarefreebq} and Propositions~\ref{propmassform}, \ref{denel},
and~\ref{unifellip}, we have 
\begin{equation}\label{impeqsec5}
\begin{array}{rcl}
\!\!\!\!\!
\displaystyle\sum_{\substack{E\in F\\H'(E)<X}}\!\!\!(\#S_2(E)-1)
&\!\!\!\!=\!\!\!\!&\displaystyle{N(V_\Z\cap S_\infty(X);2^{12}X)\prod_p \int_{S_p(F)}\frac{1}{m_p(f)}df}
+o(X^{5/6})\\
&\!\!\!\!=\!\!\!\!&\displaystyle{\frac{2^{10}}{27}\Vol(\PGL_2(\Z)\backslash\PGL_2(\R)) M_\infty(V,F;X)\prod_p\left|\frac{2^{10}}{27}\right|_p\!\!\Vol(\PGL_2(\Z_p))
M_p(V,F)}\!+\!o(X^{5/6})\\[.23in]
&\!\!\!\!=\!\!\!\!&\displaystyle{\Vol(\PGL_2(\Z)\backslash\PGL_2(\R))M_\infty(V,F;X)\prod_p\Vol(\PGL_2(\Z_p))
M_p(V,F)}+o(X^{5/6}).
\end{array}
\end{equation}
Meanwhile, Theorem \ref{thnumelip} implies that we have
\begin{equation}\label{eqhfe}
  \displaystyle\sum_{\substack{E\in F\\H'(E)<X}}1=M_\infty(F;X)\prod_pM_p(F)+o(X^{5/6}).
\end{equation}
Taking the ratio of (\ref{impeqsec5}) and (\ref{eqhfe}) now yields
Theorem~\ref{main}.
\end{proof}

To evaluate the right hand side of (\ref{eqthsec5}), we
require the following fact (see \cite[Lemma~3.1]{BK}):
\begin{lemma}\label{bk1}
Let $E$ be an elliptic curve over $\Q_p$. Then
$$\#(E(\Q_p)/2E(\Q_p))=
\left\{\begin{array}{cl}
\#E(\Q_p)[2] & {\mbox{\em if }} p\neq 2;\\[.1in]
2\cdot\#E(\Q_p)[2]& {\mbox{\em if }} p= 2.
\end{array}\right.$$
\end{lemma}
Combining Lemma \ref{bk1} with (\ref{eqmpvf}) and
(\ref{mpufdef}), we obtain that
\begin{equation}
\frac{M_p(V,F)}{M_p(F)}=\displaystyle\frac{\displaystyle\displaystyle\int_{(I,J)\in \Inv_p(F)}\frac{\#(E^{I,J}(\Q_p)/2E^{I,J}(\Q_p))}{\#E^{I,J}(\Q_p)[2]}dIdJ}
{\displaystyle\displaystyle\int_{(I,J)\in \Inv_p(F)}dIdJ}=\left\{\begin{array}{ll}
1
&\quad\mbox{if $p\neq2$;}\\[.1in]
2
&\quad\mbox{if $p=2$.}
\end{array}\right.
\end{equation}
Since we also know that $M_\infty(V,F;X)/M_\infty(F;X)=1/2$, Theorem \ref{main} implies that
\begin{eqnarray*}
\frac{\displaystyle\sum_{\substack{E\in F\\H'(E)<X}}(\#S_2(E)-1)}{\displaystyle\sum_{\substack{E\in F\\H'(E)<X}}1}\;=\;\Vol(\PGL_2(\Z)\backslash\PGL_2(\R))\prod_p\Vol(\PGL_2(\Z_p))
\end{eqnarray*}
which is then equal to $2\zeta(2)\prod_p(1-p^{-2})=2$, the Tamagawa number of $\PGL_2(\Q)$. We have proven
Theorem~\ref{ellipall} (and thus also Theorems \ref{mainellip} and~\ref{ellipcong}).

\subsection*{Acknowledgments}

We are very grateful to John Cremona, Dick Gross, Tom
Fisher, Florian Herzig, Wei Ho, Jennifer Park, Bjorn Poonen, Jerry
Wang, and the anonymous referee for their many helpful comments on an
earlier version of this manuscript. The first author was partially supported
by NSF Grant~DMS-1001828.

\end{document}